\documentclass[11pt]{article}

\usepackage{authblk}

\usepackage[utf8x]{inputenc}
\usepackage{tikz}
\usepackage{tikz-cd}
\usepackage{amsmath}
\usepackage{amssymb}
\usepackage{amsthm}
\usepackage{mathtools}
\usepackage[all,cmtip]{xy}
\numberwithin{equation}{section}
\usepackage{eucal} 

\usepackage[a4paper]{geometry}
\usepackage[export]{adjustbox}

\usepackage[colorlinks=false,pdfborder={0 0 0}]{hyperref}

\usepackage{xcolor}

\newcommand{\B}[1]{|#1|}
\DeclareMathOperator{\Img}{Im}

\newcommand{\Z}{\mathbb{Z}}

\newcommand{\R}{\mathbb{R}}
\newcommand{\C}{\mathbb{C}}
\newcommand{\A}{\mathbb{A}}
\newcommand{\G}{\mathbb{G}}
\newcommand{\Gm}{\mathbb{G}_{\mathrm{m}}}
\newcommand{\GW}{\mathrm{GW}}
\newcommand{\SH}{\mathrm{SH}}
\newcommand{\PP}{\mathbb{P}}
\newcommand{\calK}{\mathcal{K}}

\newcommand{\calC}{\mathcal{C}}
\newcommand{\calH}{\mathcal{H}}

\newcommand{\calF}{\mathcal{F}}
\newcommand{\AAA}{\mathbb{A}^1}
\newcommand{\calD}{\mathcal{D}}

\newcommand{\calS}{\mathcal{S}}
\newcommand{\frakX}{\mathfrak{X}}
\newcommand{\Spec}{\mathrm{Spec}}

\newcommand{\Lmot}{\mathrm{L}_{\mathrm{\mathrm{Mot}}}}
\newcommand{\zk}{Z_K}
\newcommand{\zka}{Z_K^{\A^1}}

\newcommand{\infsusp}{\Sigma_{\mathbb{P}^1}^\infty}

\newcommand{\colim}{colim}
\newcommand{\kmw}{\mathbf{K}^{\mathrm{MW}}}

\renewcommand{\colim}{\operatornamewithlimits{colim}}
\newcommand{\smk}{\mathrm{Sm}_k}
\newcommand{\Shnis}{\mathrm{Sh}_\mathrm{Nis}}

\newcommand{\Hk}{\mathcal{H}(k)}
\theoremstyle{plain} 
\newtheorem{proposition}[equation]{Proposition}
\newtheorem{theorem}[equation]{Theorem}
\newtheorem{corollary}[equation]{Corollary} 
 
\newtheorem{lemma}[equation]{Lemma} 
\newtheorem{thmx}{Theorem}

\theoremstyle{remark}

\theoremstyle{definition} 
\newtheorem{definition}[equation]{Definition} 
\newtheorem{example}[equation]{Example} 
\newtheorem{remark}[equation]{Remark} 

\begin{document}

\author{William Hornslien \thanks{william.hornslien@ntnu.no}}

	\title{Polyhedral products in abstract and motivic homotopy theory}

 \date{}
	\maketitle
 
 \begin{abstract}
     We introduce polyhedral products in an $\infty$-categorical setting. We generalize a splitting result by Bahri, Bendersky, Cohen, and Gitler that determines the stable homotopy type of the a polyhedral product. We also introduce a motivic refinement of moment-angle complexes and use the splitting result to compute cellular $\AAA$-homology, and $\AAA$-Euler characteristics. 
 \end{abstract}
	\section{Introduction}
Toric geometry is an important part of algebraic geometry. Since its inception in the 70s, it has grown substantially and proven its usefulness in other fields such as combinatorics, commutative algebra, and algebraic statistics. Toric geometry is the study of a certain class of algebraic varieties called toric varieties. These algebraic varieties are defined particularly nice and combinatorially, which makes it easier to do computations and prove theorems. In algebraic geometry, toric varieties are great for testing theories before proving results for larger classes of algebraic varieties. In the past decades, ways of studying toric geometry through the lens of topology have been developed. One way is by using methods from the field of toric topology, which in short, looks at the real and complex points of toric varieties as manifolds and studies their topological properties. Toric topology only considers toric varieties over the real and complex numbers, but what about different bases? Morel and Voevodsky's motivic homotopy category has made it possible to do homotopy theory with smooth algebraic varieties over any base field. In this paper we unite the two topological viewpoints and use methods from toric topology in motivic homotopy theory to study toric geometry over any field. 

The field of toric topology started with work by  Davis and Januszkiewicz \cite{DavisJanus}. They wanted to study a family of manifolds called quasi-toric manifolds. The quasi-toric manifolds were defined by simple polytopes and were homotopic to the complex points of a smooth projective toric varieties. Along with each quasi-toric manifold, they defined an auxiliary space, which we will call $\zk$, and showed that the quasi-toric manifold could be realized as orbit space of a torus (a product of $n$ circles) acting on~$\zk$. This is a topological version of what is known as Cox's construction in algebraic geometry~\cite{CoxHomogCoord}. The space $\zk$ turned out to be an interesting object in its own right, and was named the moment-angle complex by Buchstaber and Panov \cite{BuchPan2000}. They found out that $\zk$ could be described as a union of disks and circles according to the combinatorial data of a simplicial complex (see Section \ref{sec:ClassicalPolyhedral} and \ref{sec:MACs}). The quasi-toric manifolds were quotients of moment-angle complexes associated to the case where the simplicial complex was a triangulation of a sphere. In \cite{BBCG} Bahri, Bendersky, Cohen, and Gitler introduce what they call polyhedral products $(X,A)^K$ for a pair of spaces $A \subset X$  (see Section~\ref{sec:ClassicalPolyhedral} for a precise definition). Moment-angle complexes were unions of disks and circles and coincided with the polyhedral product $(D^2,S^1)^K$. There is also a closely related \emph{real moment-angle complex} $\R Z_K$, which is the polyhedral product $(D^1,S^0)^K$. Results dating all the way back to the 60s, such as a paper by Porter \cite{Porter} could be put into the framework of polyhedral products. Depending on the choice of pairs of spaces $A \subset X$, polyhedral products have connections to surprisingly many fields such as commutative algebra, geometric group theory, and robotics. A survey covering the vast applications and properties of polyhedral products can be found in \cite{PolyhedralHomotopyTheory}. 

Motivic homotopy theory takes place in Morel and Voevodsky's category of motivic spaces $\calH(k)$ over a field $k$ \cite{MorelVoevodsky99}. The category of motivic spaces is a homotopy theory for smooth schemes (algebraic varieties) over a field. See Section \ref{sec:MotivicHomotopy} for a more in-depth explanation. In motivic homotopy theory, the affine line $\AAA$ takes the role of the interval in classical homotopy theory, and is in particular contractible in $\calH(k)$. Motivic homotopy theory allows for using methods from algebraic topology for algebraic varieties, but it usually comes with some complications. One disadvantage with motivic homotopy theory is that it is in general difficult to do explicit computations. For example, the motivic cohomology of a point is only known in special cases \cite{VoevodskyCohomologyZL}. 

We introduce a motivic moment-angle complex $\zka$, which is a motivic refinement of its classical counterpart. The motivic moment-angle complex is roughly speaking a union of $\AAA$'s and $\Gm$'s, which act like disks and circles in $\calH(k)$. One could define polyhedral products directly in $\calH(k)$, but we have chosen to generalize the construction to an abstract homotopical setting. That is, in Definition \ref{def:oopolyhedralprod} we define polyhedral products in a cartesian closed $\infty$-category with small colimits in the sense of Lurie \cite{LurieHTT}. In the~$\infty$-category of topological spaces, this definition recovers the original definition. We then prove an abstract homotopical version of splitting result of Bahri, Bendersky, Cohen, and Gitler \cite{BBCG} concerning how certain polyhedral products split into a wedge of simpler pieces after a suspension (see Theorem \ref{thm:stableSplitting}). Before stating our results, we will need to introduce some notation. Fix a cartesian closed $\infty$-category $\calC$ with small colimits and let $K$ be a simplicial complex. We write $\B{K}$ for the geometric realization $K$ in $\calC$ (see Definition \ref{def:posetRealization} and \ref{def:simpcplxRealization}). By $I\not \in K$, we mean a set $I$ of vertices in $K$ such that they do not span a face in $K$. For a set of vertices $I$, we let $K_I$ be the full subcomplex of~$K$ corresponding to the vertex set $I$. The cardinality of the set $I$ is denoted by $\B{I}$. We denote suspension functor by $\Sigma$, and the smash product of two pointed objects $X,Y$ in~$\calC$ by $X\wedge Y$. The following theorem is a special case of Theorem \ref{thm:stableSplitting}.
\begin{thmx}
    \label{thm:introStableSplit}
			Let $\calC$ be a cartesian closed $\infty$-category with small colimits and fix a morphism $i\colon A \to X$ of pointed objects where $X$ is contractible. Let $K$ be a simplicial complex. Then there is an equivalence 
			\begin{equation*}
				\Sigma (X,A)^K \simeq \Sigma^2  \bigvee_{I\not \in K}  \B{K_I} \wedge A^{\wedge |I|}. 
			\end{equation*}
\end{thmx}
	The statement above for topological spaces was applied in \cite{BBCG} to describe the stable homotopy type of moment-angle complexes.  

 The motivic moment-angle complex is defined as the polyhedral product $(\AAA,\Gm)^K$, and since $\AAA$ is contractible in $\Hk$, we can apply Theorem \ref{thm:introStableSplit} to $\zka$.  
	\begin{thmx}[Theorem \ref{thm:stableMMACsplit}]
			Let $K$ be a simplicial complex. Then there is an equivalence in $\calH(k)$
			\begin{equation*}
				\Sigma \zka \simeq  \Sigma^2 \left(\bigvee_{I \not \in K} \B{K_I}\wedge \Gm^{\wedge|I|}\right) \simeq \bigvee_{I \not \in K} \B{K_I}\wedge S^{|I|+2,|I|}.
			\end{equation*}
	\end{thmx}
	We are now able to compute various motivic invariants of $\zka$. In \cite{morel2023cellular} Morel and Sawant introduced cellular $\A^1$-homology which is valued in strictly $\AAA$-invariant sheaves of abelian groups. They compute the cellular $\A^1$-homology of punctured affine spaces, projective spaces, and also low dimensional homology groups of some flag varieties. We use the stable splitting to describe the cellular $\A^1$-homology of $\zka$ in terms of~$K$. Let~$\kmw_{i}$ denote the~$i$th unramified Milnor--Witt $K$-theory sheaf. Let $ \widetilde{\mathbf{H}}_{i}(|K|)$ be the~$i$th reduced singular homology singular homology group of $|K|$ viewed as a sheaf of abelian groups.
	
	\begin{thmx}[Theorem \ref{thm:cellularHomology}]
			Let $K$ be a simplicial complex. Then $\mathbf{H}_0^{\mathrm{cell}}(\zka)=\Z$ and for $i>0$ 
            \begin{equation*}
                \mathbf{H}_i^{\mathrm{cell}}(\zka)\cong \bigoplus_{I \not \in K} \widetilde{\mathbf{H}}_{i-1}(|K_I|)\otimes \kmw_{|I|}.
            \end{equation*}
	\end{thmx}
    To prove the theorem above, one actually needs to apply Theorem \ref{thm:introStableSplit} to the derived category of chain complexes of $\AAA$-invariant sheaves. Note that this theorem also gives examples of varieties with integral torsion in their cellular $\AAA$-homology groups, e.g.\ when $K$ is a triangulation of $\R \PP^2$ (see Example \ref{ex:RP2TriangHomology}).

  There is a motivic version of the Euler characteristic called the $\AAA$-Euler characteristic. By using the stable splitting we compute the $\AAA$-Euler characteristic for $\zka$ in terms of $K$. The $\AAA$-Euler characteristic is valued in $\GW(k)$ i.e.\ the Grothendieck--Witt ring of quadratic forms over $k$.
	
	\begin{thmx}[Theorem \ref{thm:A1EulerChar}]
			The $\A^1$-Euler characteristic of the motivic moment-angle complex is 
			\begin{equation*}
				\chi_{\A^1}(\zka) =  \langle 1 \rangle - \sum_{I \not \in K} (-1)^{|I|}(\chi(K_I)-1)\cdot \langle -1 \rangle^{|I|}.
			\end{equation*}
	\end{thmx}
	The theorem above allows us to describe the $\AAA$-Euler characteristic of $\zka$ in terms of the topological Euler characteristic of full subcomplexes of $K$. In classical topology, knowing the (co)homology of a space is sufficient for determining the Euler characteristic, but this is not the case in motivic homotopy theory. When we view $\chi_{\AAA}(\zka)$ as a quadratic form over $\R$, the rank of $\chi_{\AAA}(\zka)$ is~$\chi(\zk)$ and the signature of  $\chi_{\AAA}(\zka)$ is~$\chi(\R\zk)$. This result examplifies why $\zka$ is a motivic refinement of $Z_K$ and $\R\zk$. 
 
 The paper is structured as follows. In Section \ref{sec:polyhedralProductFunctor} the classical definition of a polyhedral product and various classical results are recalled. In Section \ref{sec:inftysetup}, we define polyhedral products in~$\infty$-categories. Theorem \ref{thm:introStableSplit} is then proven in Section \ref{sec:stableSplit}. In Section \ref{sec:exampleCats}, we review polyhedral products in equivariant and motivic homotopy theory. We also define motivic moment-angle complexes. In Section \ref{sec:affine}, we give various smooth models of motivic moment-angle complexes, and briefly study their connection to toric varieties. Section~\ref{sec:motivicInvariants} is dedicated to computing various invariants of the motivic moment-angle complexes.
     \subsection*{Acknowledgements}
	The author would like to thank Markus Szymik for many discussions, feedback along the way. The author thanks Gereon Quick for comments on the manuscript. The author would also like to thank Fernando Abellan, Marc Hoyois, Abigail Linton, Louis Martini, Clover May, Fabien Morel,  Marius Nielsen, Paul Arne Østvær, Oliver Röndigs, and Markus Spitzweck for helpful discussions regarding the project.  The author thanks the project \emph{Pure Mathematics in Norway} funded by the Trond Mohn Foundation for economic support.

    \subsection*{Conventions}
    By a topological space we mean a $CW$-complex. Starting in Section \ref{sec:polyhedralProductFunctor}, we will use the language of $\infty$-categories as developed by Lurie in \cite{LurieHTT}. By $\infty$-category, we mean an~$(\infty,1)$-category. When necessary, we will implicitly view $1$-categories as $\infty$-categories through the nerve embedding. From Section \ref{sec:exampleCats}, we will assume $k$ to be  a perfect field of characteristic not equal to $2$. 
	\section{Polyhedral products}
 \label{sec:polyhedralProductFunctor}

 In this section we give a brief overview of some core properties of the polyhedral product functor and review some classical results. 
 \subsection{Preliminaries on simplicial complexes}

 \begin{definition}
     Let $m$ be a positive integer. An abstract simplicial complex $K$ is a family of subsets of $[m]:= \{1, \ldots, m\}$ that is closed under taking subsets. In more geometric terms, $K$ is a simplicial complex with $m$ vertices labeled by the set~${[m]=\{1,\ldots, m\}}$. A~$(n-1)$-face $\sigma$ of $K$ is given by a subset $\sigma = \{i_1, \ldots, i_n\}$ with $1\leq i_1 < \ldots < i_n \leq m$. All subsets $\tau \subset \sigma$ define faces in $K$ as well. In particular, $K$ includes the empty face $\emptyset$. 
 \end{definition}
Whenever we say simplicial complex, we mean an abstract simplicial complex. We will now define an important family of subcomplexes. 
\begin{definition}
    Let $K$ be a simplicial complex and $I\subset [m]$. The full subcomplex $K_I$ consists of all faces of $K$ that have their vertex set as a subset of $I$, i.e.\  
    \begin{equation*}
        K_I := \{\sigma \cap I | \sigma \in K\}.
    \end{equation*}
\end{definition}
Simplicial complexes can be seen as topological spaces, this is done by geometric realization. 
\begin{definition}
\label{def:RealizationSpacesK}
    Denote the geometric realization of $K$ as a topological space by $\B{K}$.
\end{definition}
 A simplicial complex comes with the natural structure of a poset. That is, the face poset ordered by inclusion. Each face represents an object and if $\sigma$ is a subface of $\tau$, then $ \sigma < \tau$ in the poset. To this poset is an associated category, which will be essential going forward.
\begin{definition}
       Let $K$ be a simplicial complex. The \emph{face poset category} $\calK$ is defined as follows. The objects of $\calK$ are given by the simplices of $K$, including an initial object~$\emptyset$ which corresponds to the empty face. Let $\sigma, \tau \in K$ be two simplices of $K$. If $\sigma$ is a subface of $\tau$, then there is a unique morphism $f_{\sigma \leq \tau }\colon \sigma \to \tau$. Let $I \subset [m]$, we denote the face poset category of $K_I$ by $\calK_I$.
\end{definition}
   The following two constructions are central in combinatorics. 

   \begin{definition}
   \label{def:StanleyReisnerRing}
       Let $\mathbf{k}$ be a ring. For a simplicial complex $K$, we define the \emph{Stanley--Reisner ideal} $I_K$ as the square-free monomial ideal corresponding to non-faces of $K$, i.e.\
  \begin{equation*}
      I_K = (x_{i_1}  \ldots x_{i_r} | \{i_1, \ldots ,i_r\} \not \in K).
  \end{equation*}
  We define the \emph{Stanley--Reisner ring} as the quotient
  \begin{equation*}
      \mathbf{k}[K] := \mathbf{k}[x_1,\ldots, x_m]/I_K.
  \end{equation*}
   \end{definition}
   The Stanley--Reisner ring has connections to fields such as toric geometry, polytopes, and splines \cite[Chapter III]{StanleyCombinatoricsComalg}. 
    
    \begin{definition}
        To a simplicial complex $K$, the \emph{Alexander dual} $K^\vee$ is the simplicial complex whose faces are complements of non-faces of $K$, i.e.\ \begin{equation*}
      K^\vee := \{ \sigma \in [m] | [m] \setminus \sigma \not \in K \}.
  \end{equation*}
    \end{definition}
    See \cite[Example 2.26 and Corollary 2.28]{BuchPan3} for further results on Alexander duality of simplicial complexes. 
  
\begin{example}
\label{ex:squareCplx}
    Let $K$ be the boundary of a square, i.e.\
    \begin{equation*}
        K = \{\emptyset, \{1\},\{2\},\{3\},\{4\},\{1,3\},\{2,3\},\{2,4\},\{1,4\}\}.
        \end{equation*}
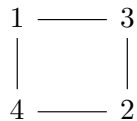
        \begin{equation*}
\begin{tikzcd}
1 \arrow[d, no head] \arrow[r, no head] & 3 \arrow[d, no head] \\
4 \arrow[r, no head]                    & 2                   
\end{tikzcd}
    \end{equation*}
    Then 
    \begin{equation*}
        \mathbf{k}[K]=\mathbf{k}[x_1,x_2,x_3,x_4]/(x_1x_2,x_3x_4)
    \end{equation*}
    and the Alexander dual is the simplicial complex 
    \begin{equation*}
        K^\vee = \{\emptyset, \{1\},\{2\},\{3\},\{4\},\{1,2\},\{3,4\}\}.
    \end{equation*}
    \begin{equation*}
\begin{tikzcd}
1 \arrow[d, no head] & 3 \arrow[d, no head] \\
2                    & 4                   
\end{tikzcd}
    \end{equation*}
    Note the swap of position of vertices 2 and 4 in the picture. 
\end{example}
\subsection{The classical construction}
\label{sec:ClassicalPolyhedral}
	Polyhedral products were first defined by Bahri, Bendersky, Cohen, and Gitler in \cite{BBCG}. A good survey of the work done and its connections to other fields can be found in \cite{PolyhedralHomotopyTheory}. Let 
 \begin{equation*}
     (\underline{X},\underline{A})=((X_1,A_1) ,\ldots , (X_m,A_m))
 \end{equation*} be a sequence of $m$ pairs of pointed topological spaces $A_i \subset X_i$. Let $K$ be a simplicial complex on the vertex set $[m]=\{1, \ldots, m\}$. The polyhedral product $(\underline{X},\underline{A})^K$ is defined as the following union
	
	\begin{equation*}
		(\underline{X},\underline{A})^K= \bigcup_{\sigma \in K}D(\sigma) \subset \prod_{i=1}^m X_i,
	\end{equation*}
	where 
	\begin{equation*}
		D(\sigma) = \prod_{i=1}^m Y_i \quad \text{where} \quad Y_i=\begin{cases}X_i & \text{if } i \in \sigma, \\ A_i & \text{if } i \not\in \sigma.\end{cases}
	\end{equation*}
	\begin{remark}
	    When all pairs $(X_i,A_i)$ are the same pair and $K$ is a simplicial complex, we will write $(X,A)^K$ for the associated polyhedral product.
	\end{remark}
 A wide array of spaces can be created as polyhedral products. 
	\begin{example}
		
		Let $K$ be a simplicial complex with two disjoint vertices. \label{PolyhedraproductEx}
		\begin{enumerate}
			\item $(D^1,S^0)^K = D^1 \times S^0 \cup S^0 \times D^1 \simeq S^1$ 
			\item $(D^2,S^1)^K = D^2 \times S^1 \cup S^1 \times D^2 \simeq S^3$ 
			\item $(S^1,*)^K  = S^1 \times * \cup * \times S^1 \simeq S^1\vee S^1$
		\end{enumerate}
	\end{example}
 \begin{example}
		
		Let $K$ be a disjoint union of $m$ points. Then there is a homotopy equivalence $(\underline {X},*)^K = X_1 \vee \ldots \vee X_m.$

	\end{example}
 In this case, one can observe how $(\underline {X},*)^K$ interpolates between $X_1 \times \ldots \times X_m$ and~$X_1 \vee \ldots \vee X_m$ when $K$ ranges from a full $(m-1)$-simplex to $m$ disjoint points. 
	The following proposition tells us how certain operations with simplicial complexes affect the polyhedral products.
	\begin{proposition}[{\cite[Proposition 4.2.5]{BuchPan}}]
		Let $K$ and $K'$ be two simplicial complexes and let $K\star K'$ denote their join. Then $(X,A)^K \times (X,A)^{K'} \simeq (X,A)^{K\star K'}.$
	\end{proposition}
	\begin{example}
		
		\label{PolyhedraproductExTorus}
		Let $K_1 = \{\emptyset, \{1\}, \{2\}\}$, $K_2 = \{\emptyset, \{3\}, \{4\}\}$ be simplicial complexes. In this case $K_1$ and $K_2$ are complexes consisting of two disjoint points. Let $K = K_1 \star K_2$, then $K$ is a complex  shaped like the boundary of a square as in Example \ref{ex:squareCplx}. 
		\begin{enumerate}
			\item $(D^1,S^0)^K \simeq S^1\times S^1$
			\item $(D^2,S^1)^K  \simeq S^3\times S^3$
		\end{enumerate}
	\end{example}
	
	\begin{example}
		Let $K = \partial \Delta^n$.
		\begin{enumerate}
			\item $(D^1,S^0)^K \simeq S^n$
			\item $(D^2,S^1)^K \simeq S^{2n+1}$
		\end{enumerate}
	\end{example}
	\begin{example} [{\cite[Corollary 9.7] {GrbicTheriault07}}]
		Let $K$ be the disjoint union of $m$ points. Then there is a homotopy equivalence $(D^2,S^1)^K \simeq \bigvee_{l=2}^m (S^{l+1})^{\vee (l-1)\begin{psmallmatrix}m \\ l\end{psmallmatrix}}$.
	\end{example}
 The following two results are due to Bahri, Bendersky, Cohen, and Gitler. In particular, they show how certain polyhedral products split into a wedge of nice pieces after suspending. Let $\Sigma$ denote the reduced suspension functor.  When $I = \{i_1, \ldots, i_n\} \subset [m]$ and $Y_1, \ldots, Y_m$ are pointed topological spaces, we will write $\widehat{Y}^I := Y_{i_1} \wedge \ldots \wedge Y_{i_k}$.
 \begin{theorem}[{\cite[Theorem 2.15]{BBCG}}]
    Let $K$ be a simplicial complex on the vertex set~$[m]$ and let $(\underline{X},\underline{A})$ be a family of pairs. If $A_i$ is contractible for each $1\leq i \leq m$, then
    \begin{equation*}
        \Sigma (\underline{X},\underline{A})^K \simeq \Sigma \bigvee_{I \in K}\widehat{X}^I.
    \end{equation*}
 \end{theorem}
 There is also a version of the splitting when all the $X_i$'s are contractible. 
 \begin{theorem}[{\cite[Theorem 2.21]{BBCG}}]
    \label{thm:classicalStableSplit}
    Let $K$ be a simplicial complex on the vertex set~$[m]$ and let $(\underline{X},\underline{A})$ be a family of pairs. If $X_i$ is contractible for each $1\leq i \leq m$, then
    \begin{equation*}
        \Sigma (\underline{X},\underline{A})^K \simeq \Sigma \bigvee_{I \not \in K}\B{K_I} \star \widehat{A}^I.
    \end{equation*}
 \end{theorem}
 In particular, the two splitting results above are special cases of a theorem that requires the map $A_i \to X_i$ to be null-homotopic for all $i$ \cite[Theorem 2.13]{BBCG}. 

 In \cite{Davis2012} Davis finds a general formula for the Euler characteristic of the polyhedral product. 
		\begin{theorem}[\cite{Davis2012}]
			\label{thm:DavisEuler}
			Let $K$ be a simplicial complex on the vertex set~$[m]$ and $A \subset X$ be two finite CW-complexes. Then
			\begin{equation*}
				\chi((X,A)^K)=\sum_{\sigma \in K} (\chi (X)-\chi(A))^{|\sigma|}\chi(A)^{m-|\sigma|}.
			\end{equation*}
		\end{theorem}

\begin{example}
\label{ex:EulerCharEx}
    Let $K$ be a simplicial complex and $X$ a finite $CW$-complex.
    \begin{enumerate}
        \item $\chi((D^2,S^1)^K) = 0$
        \item $\chi((D^1,S^0)^K) = \Sigma_{\sigma \in K}(-1)^{|\sigma|}\cdot 2^{m-|\sigma|}$
        \item $\chi((X,*)^K)=\Sigma_{\sigma \in K} (\chi(X)-1)$
    \end{enumerate}
    
\end{example}
 \subsection{Moment-angle complexes}
        \label{sec:MACs}
		A particularly well studied family of polyhedral products are the \emph{moment-angle complexes}. For a simplicial complex $K$, the associated moment-angle complex is defined as the polyhedral product $\zk := (D^2,S^1)^K$. In fact, moment-angle complexes were studied long before polyhedral products were defined. The original construction goes back to Davis and Januskiewicz in \cite[§4.1]{DavisJanus}, but only for a family of moment-angle complexes known as moment-angle manifolds. However, the definition that inspired polyhedral products, i.e.\ moment-angle complexes viewed as a union of products, is due to Buchstaber and Panov \cite[Definition~6.10]{BuchPan3}. Since Buchstaber and Panov's reinterpretation of moment-angle complexes, a lot of progress has been made. We will now present some results that both will be of use later, and emphasise how moment-angle complexes are studied. 

Since $D^2$ is contractible, we can use Theorem \ref{thm:classicalStableSplit} to describe $\Sigma \zk$. There is a homotopy equivalence 
\begin{equation*}
            \Sigma \zk \simeq \bigvee_{I \not \in K} \Sigma^{|I|+2}|K_I|.
        \end{equation*}
This decomposition makes it possible to describe the homology of $\zk$ in terms of full subcomplexes of $K$ for any homology theory. In particular, it describes the singular cohomology groups of $\zk$. The cohomological ring structure can also be described. In~\cite{BuchPan3}, the cohomology of $\zk$ is computed using the Eilenberg--Moore spectral sequence. For a ring $\mathbf{k}$, recall that $\mathbf{k}[K]$ is the Stanley--Reisner ring as defined in Definition \ref{def:StanleyReisnerRing}. 

\begin{theorem}[{\cite[Theorem 7.6]{BuchPan3}}]
        \label{thm:TorIso}
        Let $K$ be a simplicial complex on the vertex set~$[m]$. Let $\mathbf{k}$ be a field or $\Z$. There is an isomorphism of algebras
            \begin{equation*}
            H^{2j-i}(\zk;\mathbf{k}) \cong  \mathrm{Tor}^{-i,2j}_{\mathbf{k}[x_1,\ldots,x_m]}(\mathbf{k},\mathbf{k}[K]).
        \end{equation*}
\end{theorem}
   
        The isomorphism of Theorem \ref{thm:TorIso} gives a natural bigrading on the cohomology ring of $\zk$. The wedge decomposition of $\Sigma \zk$ splits the cohomology of $\zk$ into nice subgroups given by the cohomology of full subcomplexes of $K$. A result by Baskakov makes it possible to describe the ring structure with respect to the full subcomplexes~$K_I$. 
\begin{theorem}[{\cite[Theorem 1]{Baskakov02}}, {\cite[Theorem 1]{BaskBuchPan}}]
\label{thm:cupProductMAC}
            Let $K$ be a simplicial complex on $m$ vertices. Let $\mathbf{k}$ be a field or $\Z$.  There is an isomorphism of groups 
            \begin{equation*}
                H^i(\zk;\mathbf{k}) \cong \begin{cases}
                    \mathbf{k} & i = 0, \\
                    \displaystyle \bigoplus_{I \not \in K}  \widetilde{H}^{i-|I|-1}(K_I; \mathbf{k}) & i > 0.
                \end{cases}
            \end{equation*}
            In particular, there is an isomorphism of algebras 
            \begin{equation*}
                H^*(\zk;\mathbf{k}) \cong \mathbf{k}\oplus \bigoplus_{I \not \in K} \widetilde{H}^*(K_I; \mathbf{k}).
            \end{equation*}                 
            The products in the sum on the right are given as follows: for $I, J \not \in K$, with $I \cap J = \emptyset$, let~$\alpha \in \widetilde{H}^{p}(K_I;\mathbf{k})$ and $\beta \in \widetilde{H}^{q}(K_J;\mathbf{k})$ be nontrivial cohomology classes. Then there exists a nontrivial cohomology class $\gamma \in \widetilde{H}^{p+q}(K_{I\cup J};\mathbf{k})$ such that $\alpha \smile \beta = \gamma$. All products of cohomology classes in $H^*(\zk;\mathbf{k})$ arise in this way. 
            
        \end{theorem}

 It is also possible to describe Massey products \cite{GrbicLinton} and Steenrod operations \cite{agarwal2024steenrod} in terms of full subcomplexes of $K$.   
		
		Moment-angle complexes also share a deep connection to toric geometry, since the complex points of any smooth projective toric variety can be realized as the orbits of a torus acting on the polyhedral product $(\C, \C^\times)^K$. Note that we have a homotopy equivalence $\zk \simeq (\C, \C^\times)^K$. This is a topological version of something known as Cox construction \cite[Theorem 2.1]{CoxHomogCoord} of smooth projective algebraic varieties in algebraic geometry. See \cite[§5.4]{BuchPan} for a description of the Cox construction from a toric topological viewpoint.

  There is also related polyhedral product called the \emph{real moment-angle complexes} defined as $\R Z_K:=(D^1,S^0)^K$. The real moment-angle complex can be seen as the fixed points of a $C_2$-action on $\zk$. Theorem \ref{thm:classicalStableSplit} also makes it possible to describe $\R\zk$,
         \begin{equation*}
            \Sigma \R Z_K \simeq \bigvee_{I \not \in K} \Sigma^2 |K_I|.
        \end{equation*}
  This makes it straightforward to describe the cohomology groups of $\R\zk$. The ring structure has been computed  \cite{CaiProducts}, but is far more complicated than the case of the moment-angle complex. 
		\section{$\infty$-categorical setup}
  \label{sec:inftysetup}
		We will now consider an $\infty$-categorical version of polyhedral products. From now on we freely use the language of $\infty$-categories as developed by Lurie in \cite{LurieHTT}. Let $\calC$ be a cartesian closed $\infty$-category with finite colimits. That is, for each~$X \in {\calC}$ the product functor $X\times {-}$ has a right adjoint, and thus preserves all colimits. Denote the terminal object of $\calC$ by~$*$. A pointed object in $\calC$ is an object $X$ together with a map~$x\colon *\to X$. From now on, when we talk about colimits, we mean in the~$\infty$-categorical sense. Whenever $\calC$ is the nerve of a model category,  computing the homotopy colimit in the underlying model category suffices. 
        \begin{definition}
            The \emph{suspension} of an object $X \in \calC$ is the pushout 
            \begin{equation*}
                \begin{tikzcd}
        X \arrow[dr, phantom, "\scalebox{1}{$\ulcorner$}" , very near end, color=black]\ar[r]\ar[d] & * \ar[d] \\
        * \ar[r] & \Sigma X  .
    \end{tikzcd}
            \end{equation*}
        \end{definition}
         \begin{definition}
         \label{def:wedge}
            The \emph{wedge} of two pointed objects  $X,Y \in \calC$ is the pushout 
            \begin{equation*}
                \begin{tikzcd}
        * \arrow[dr, phantom, "\scalebox{1}{$\ulcorner$}" , very near end, color=black]\ar[r]\ar[d] & X \ar[d] \\
        Y \ar[r] & X\vee Y  .
    \end{tikzcd}
            \end{equation*}
        \end{definition}
   \begin{definition}
            The \emph{smash product} of two pointed objects  $X,Y \in \calC$ is the pushout 
            \begin{equation*}
                \begin{tikzcd}
        X\vee Y \arrow[dr, phantom, "\scalebox{1}{$\ulcorner$}" , very near end, color=black]\ar[r]\ar[d] & X\times Y \ar[d] \\
        * \ar[r] & X\wedge Y  .
    \end{tikzcd}
            \end{equation*}
        \end{definition}
        \begin{definition}
            The \emph{join} of two pointed objects  $X,Y \in \calC$ is the pushout 
            \begin{equation*}
                \begin{tikzcd}
        X \times Y \arrow[dr, phantom, "\scalebox{1}{$\ulcorner$}" , very near end, color=black]\ar[r]\ar[d] & X  \ar[d] \\
        Y \ar[r] & X\star Y  .
    \end{tikzcd}
            \end{equation*}
        \end{definition}
        
        \begin{lemma}[{\cite[Lemma 3.5]{lavenir2023hiltonmilnors}}]
            Let $X,Y \in C$ be pointed. There is an equivalence 
            \begin{equation*}
                X\star Y \simeq \Sigma X \wedge Y.
            \end{equation*}
        \end{lemma}
        \begin{definition}
            Let $\calC$ be an $\infty$-category and let  \begin{equation*}
     (\underline{X},\underline{A})=((X_1,A_1) ,\ldots , (X_m,A_m))
 \end{equation*}  be a sequence of pairs of pointed objects in $\calC$ equipped with a map $\iota_i\colon A_i \to X_i$. We call $(\underline{X},\underline{A})$ a \emph{family of pairs}.
        \end{definition}
		\begin{definition}
  \label{def:oopolyhedralprod}
			Let $(\underline{X},\underline{A})$ be a family of pairs and let $K$ be a simplicial complex on the vertex set $[m]$.  Let $\calK$ be the face poset category of $K$ ordered by inclusions, that is~$\sigma > \tau$ if $\sigma \subsetneq \tau$. We define the polyhedral product $(\underline{X},\underline{A})^K$ as
			\begin{equation*}
				(\underline{X},\underline{A})^K := \colim_{\sigma \in K} D(\sigma),
			\end{equation*}
			with $D(\sigma)$ defined as follows:
			\begin{equation*}
				D(\sigma) = \prod_{i=1}^m Y_i \quad \text{where} \quad Y_i=\begin{cases}X_i & \text{if } i \in \sigma, \\ A_i & \text{if } i \not\in \sigma.\end{cases}
			\end{equation*}
			For any pair of simplices $\sigma \subset \tau \in K$ the map from $D(\sigma)$ to $D(\tau)$ is induced by the products of the maps~$\iota_i$ and the identity. In other words, $(\underline{X},\underline{A})^K$ is the colimit of the diagram 
   \begin{equation*}
       D \colon \calK \to \calC.
   \end{equation*}
		\end{definition}
		\begin{example}
			\begin{enumerate}
				\item Suppose that each $A_i$ is the terminal object $*$. If $K$ is the disjoint union of $m$ points then $(\underline{X},\underline{*})^K$ is $X_1 \vee X_2 \vee \ldots \vee X_m$.
				\item Suppose that $K$ is the $(m-1)$-simplex then $(\underline{X},\underline{A})^K$ is $X_1 \times X_2 \times \ldots \times X_m$.
				\item Suppose that $K$ is the complex of two disjoint vertices and that each $X_i \simeq *$ then $(\underline{X},\underline{A})^K \simeq  \Sigma A_1 \wedge A_2 \simeq A_1 \star A_2 $, the join of $A_1$ and $A_2$.
			\end{enumerate}
		\end{example}
   The following proposition is an $\infty$-version of Proposition \ref{prop:JoinProductClassical} does.
		\begin{proposition}
            \label{prop:JoinProductClassical}
		     Suppose $K$ and $K'$ are two simplicial complexes, and denote their join by $K \star K'$, then $(\underline{X},\underline{A})^K \times (\underline{X},\underline{A})^{K'} = (\underline{X},\underline{A})^{K \star K'}$.
		\end{proposition}
  \begin{proof}
      Since $\calC$ is cartesian closed, cartesian products preserve colimits. Thus there is there is a chain of equivalences 
      \begin{equation*}
          (\underline{X},\underline{A})^K \times (\underline{X},\underline{A})^{K'} \simeq  \colim_{\sigma' \in K'} \left((\underline{X},\underline{A})^K  \times D(\sigma')\right) \simeq \colim_{\sigma' \in K'} \left( \colim_{\sigma \in K} \left( D(\sigma) \times D(\sigma')  \right)\right).
      \end{equation*}
      The iterated colimits can be rewritten as one colimit iterating over $\sigma \in K$ and $\sigma' \in K'$. This yields the equivalences 
      \begin{equation*}
           \colim_{\sigma' \in K'} \left( \colim_{\sigma \in K} \left( D(\sigma) \times D(\sigma')  \right)\right) \simeq \colim_{\sigma \in K, \sigma' \in K'} \left(D(\sigma) \times D(\sigma') \right)\simeq (\underline{X},\underline{A})^{K * K'}. \qedhere
      \end{equation*}
  \end{proof}
		We will also need a space called the polyhedral smash product. It was first defined for topological spaces in \cite{BBCG}. 
		\begin{definition}
			\label{def:smashpolyhedral}
			Let $(\underline{X},\underline{A})$ be a family of pairs and let $K$ be a simplicial complex on the vertex set $[m]$.  Let $\calK$ be the face poset category of $K$ ordered by inclusions. We define the \emph{polyhedral smash product} $\widehat{(\underline{X},\underline{A})}^K$ as
			\begin{equation*}
				\widehat{(\underline{X},\underline{A})}^K := \colim_{\sigma \in K} \widehat{D}(\sigma),
			\end{equation*}
			with $\widehat{D}(\sigma)$ defined as follows:
			\begin{equation*}
				\widehat{D}(\sigma) = \bigwedge_{i=1}^m Y_i \quad \text{where} \quad Y_i=\begin{cases}X_i & \text{if } i \in \sigma, \\ A_i & \text{if } i \not\in \sigma.\end{cases}
			\end{equation*}
			For any pair of simplices $\sigma \subset \tau \in K$ the map from $\widehat{D}(\sigma)$ to $\widehat{D}(\tau)$ is induced by the maps~$\iota_i$ and the identity. 
		\end{definition}

		\begin{definition}
			Let $(\underline{X},\underline{A})$ be a family of pairs, $K$ a simplicial complex, and $I \subset [m]$. Define 
			\begin{equation*}
				(\underline{X}_I,\underline{A}_I)= ((X_{i_j},A_{i_j}))_{j=1}^{|I|}
			\end{equation*}
			as the subfamily of $(\underline{X},\underline{A})$ determined by $I$.
            We define $(\underline{X},\underline{A})^{K_I} := (\underline{X}_I,\underline{A}_I)^{K_I}$ and similarly for the polyhedral smash product $ \widehat{(\underline{X},\underline{A})}^{K_I}$. 
		\end{definition}
    We will now see how pushouts of simplicial complexes induce pushouts of polyhedral products. Let $K$ be a simplicial complex, and suppose there exists subcomplexes $K_1, K_2$, and $L$ such that $K = K_1 \cup_L K_2$. To be able to relate the various polyhedral products, it is important that $K, K_1, K_2$, and $L$ are all on the same vertex set. If $K$ is a simplicial complex on the vertex set $[m]$ we write $\overline{K_1}, \overline{K_2}, \overline{L}$ for the simplicial complexes $K_1, K_2$, and $L$ regarded as simplicial complexes on the vertex set $[m]$. 
    \begin{proposition}
        \label{prop:PushoutPolyProd}
        Let $K$ be a simplical complex on the vertex set $[m]$ with subcomplexes~$K_1, K_2$, and $L$ such that $K = K_1 \cup_L K_2$. Let $(\underline{X},\underline{A})$ be a family of $m$ pairs. Then there is a pushout of polyhedral products
        \begin{equation*}
                \begin{tikzcd}
        (\underline{X},\underline{A})^{\overline{L}} \arrow[dr, phantom, "\scalebox{1}{$\ulcorner$}" , very near end, color=black]\ar[r]\ar[d] & (\underline{X},\underline{A})^{\overline{K_1}}   \ar[d] \\
        (\underline{X},\underline{A})^{\overline{K_2}}  \ar[r] & (\underline{X},\underline{A})^{K}.  
    \end{tikzcd}
            \end{equation*}
    \end{proposition}
    \begin{proof}
        Denote the face categories of $K_1, K_2$, and $L$ by $\calK_1, \calK_2$, and $\mathcal{L}$. Let $D_1, D_2$ and~$D_L$ be the diagram $D$ restricted to $\calK_1, \calK_2$, and $\mathcal{L}$ respectively. Hence, the colimits of $D_1, D_2$, and~$D_L$ are $(\underline{X},\underline{A})^{\overline{K_1}}, (\underline{X},\underline{A})^{\overline{K_2}}$, and $(\underline{X},\underline{A})^{\overline{L}}$ respectively. Let  $D'$ be the following diagram of diagrams
        \begin{equation*}
            D_1 \leftarrow D_L \rightarrow D_2.
        \end{equation*}
        This diagram is a left Kan extension the diagram $D$, and hence has the same colimit as~$D$, which is $(\underline{X},\underline{A})^{K}$. By \cite[Proposition 4.4.2.2]{LurieHTT}, we may compute the colimits termwise in the diagram, which yields the desired pushout square.
    \end{proof}
    Let $K,L$ be simplicial complexes. We denote the disjoint union by $K\sqcup L$.  
            \begin{corollary}
                Let $K_1,K_2$ be simplicial complexes on the vertex sets $[m]$ and $[n]$. Let~$(\underline{X},\underline{*})$ be a family of $m+n$ pairs.
                \begin{equation*}
                    (\underline{X},\underline{*})^{K_1 \sqcup K_2} \simeq (\underline{X}_{\{1,\ldots,m\}},\underline{*})^{K_1} \vee (\underline{X}_{\{m+1,\ldots,m+n\}},\underline{*})^{K_2}.
                \end{equation*}
            \end{corollary}
		\begin{remark}
		    A version of Proposition \ref{prop:PushoutPolyProd} in the $\infty$-category of spaces was used by Grbi\'c and Theriault in \cite{GrbicTheriault07} to determine that moment-angle complexes have the homotopy type of a wedge of spheres when $K$ is a shifted simplicial complex.
		\end{remark}
		
		\section{Stable splitting of polyhedral products}
  \label{sec:stableSplit}
		In this section we will prove some stable splitting results for polyhedral products. The results are $\infty$-categorical generalizations of work by Bahri, Bendersky, Cohen, and Gitler in \cite{BBCG}. For the rest of this section, unless stated otherwise, fix $\calC$ to be a cartesian closed~$\infty$-category. 
  \begin{definition}
  \label{def:posetRealization}
      For a poset category $\calD$, let $\B{\calD}$ be the realization of $\calD$ in $\calC$. That is, the colimit over the $\calD$-shaped diagram, with constant value $* \in \calC$.
  \end{definition}
\begin{definition}
			Let $\calD$ be a poset category. For an object $a \in \calD$ let $\calD_{\leq a}$ be the undercategory of $a$. Let $\calD_{<a}$ be the category of objects in $\calD)$ that are strictly smaller than $a$.
    \end{definition}
  \begin{definition}
  \label{def:simpcplxRealization}
      When $K$ is a simplicial complex and $\calK$ is its face category, we will write $\B{K}  := \B{\calK_{<\emptyset}}$. 
  \end{definition}
  \begin{remark}
      When $\calC$ is the $\infty$-category of spaces, the realization $\B{K}$ of Definition \ref{def:simpcplxRealization} and the geometric realization of $K$ as a topological space from Definition \ref{def:RealizationSpacesK} agrees. 
  \end{remark}
   The main result of this section is the following theorem. 
		\begin{theorem}
			\label{thm:stableSplitting}
			Let $K$ be a simplicial complex with $m$ vertices and let $(\underline{X},\underline{A})$ have the property that each map $\iota_i \colon A_i \to X_i$ is null. Then there is an equivalence 
			\begin{equation*}
				\Sigma (\underline{X},\underline{A})^K \simeq \Sigma \left( \bigvee_{I\subset [m]}  \left( \bigvee_{\sigma \in K_I} \B{(\calK_I)_{<\sigma}}\star\widehat{D}(\sigma)\right) \right). 
			\end{equation*}
		\end{theorem}
		We will postpone the proof of the theorem until the end of the section. When each~$X_i$ is contractible Theorem \ref{thm:stableSplitting} simplifies to the following.
		\begin{corollary}
			\label{cor:stableSplittingContractible}
			Let $K$ be a simplicial complex with $m$ vertices and let $(\underline{X},\underline{A})$ be a family of pairs where each~$X_i$ is contractible. Then there is an equivalence 
			\begin{equation*}
				\Sigma (\underline{X},\underline{A})^K \simeq \Sigma  \bigvee_{I\not \in K}  \B{K_I}\star \widehat{A}^I. 
			\end{equation*}
		\end{corollary}
		\begin{remark}
			Proposition \ref{prop:stableSmashProdSplit} and Theorem~\ref{thm:stableSplitting} are $\infty$-categorical versions of Theorem~2.10 and 2.13 in \cite{BBCG}. 
		\end{remark}
		To prove Theorem \ref{thm:stableSplitting} and Corollary \ref{cor:stableSplittingContractible} a collection of results will be needed. We will follow the proof strategy of \cite{BBCG}. The following result is known as the Ganea splitting. 
		\begin{lemma}[{\cite[Corollary 2.24.2]{Devalapurkar2021}}]
			\label{lem:ganeasplitting}
			Let $\calC$ be an  $\infty$-category with finite limits and pushouts. Then, for every pair of pointed objects $X,Y \in \calC$, there is a natural equivalence~$\Sigma (X \times Y) \simeq \Sigma X \vee \Sigma Y \vee \Sigma(X \wedge Y)$.
		\end{lemma}
		The stable splitting of a product gives a nice description of the stable splitting of larger products. Consecutive applications of the Ganea splitting yields the following result.
		\begin{corollary}
			\label{lem:iteratedGanea}
			Let $Y_i$ be pointed objects in $\calC$. There is an equivalence
			\begin{equation*}
				\Sigma (Y_1 \times \ldots \times Y_m) \simeq \Sigma \bigvee_{I \subset [m]} \widehat{Y}^I.
			\end{equation*}
		\end{corollary}
		Recall that $\widehat{(\underline{X},\underline{A})}^K$ denotes the polyhedral smash product from Definition \ref{def:smashpolyhedral}. 
		
		\begin{proposition}
			\label{prop:stableSmashProdSplit}
			Given a simplicial complex $K$ with $m$ vertices and a family of pairs~$(\underline{X},\underline{A})$, we have the following natural equivalence
			\begin{equation*}
				\Sigma(\underline{X},\underline{A})^K \simeq \Sigma \left( \bigvee_{I\subset [m] } \widehat{(\underline{X},\underline{A})}^{K_I}\right).
			\end{equation*}
		\end{proposition}
		\begin{proof}
			For $I \in [m]$ and  $\sigma \in K$, define
			\begin{equation*}
				\widehat{D}_I(\sigma):= \bigwedge_{i\in I} Y_i \quad \text{where} \quad Y_i=\begin{cases}X_i & \text{if } i \in \sigma, \\ A_i & \text{if } i \not\in \sigma.\end{cases}
			\end{equation*}
			Since suspending commutes with colimits, there is an equivalence
   \begin{equation*}
       \Sigma(\underline{X},\underline{A})^K \simeq \Sigma \colim_{\sigma \in K}  D(\sigma) \simeq \colim_{\sigma \in K}  \Sigma D(\sigma).
   \end{equation*}
   
   By Corollary \ref{lem:iteratedGanea} we can describe $\Sigma D(\sigma)$ for $\sigma \in K$. For each $\sigma \in K$, there is an equivalence
   \begin{equation*}
       \Sigma D(\sigma) = \Sigma \bigvee_{I \subset [m]} \widehat{D}_I(\sigma).
   \end{equation*}
   Since the equivalence of Corollary \ref{lem:iteratedGanea} was natural and colimits commute, there is an equivalence 
			\begin{equation*}
				\Sigma(\underline{X},\underline{A})^K  \simeq  \colim_{\sigma \in K} \left( \Sigma \bigvee_{I\subset [m]}  \widehat{D}_I(\sigma) \right) \simeq \Sigma \left( \bigvee_{I\subset [m]} \colim_{\sigma \in K} \widehat{D}_I(\sigma) \right).
			\end{equation*}
			Now fix some $I\subset [m]$ and consider the colimit $\displaystyle \colim_{\sigma \in K}  \widehat{D}_I(\sigma)$. For each $i \not \in I$, the maps induced by $\tau \subset \sigma$ where $\tau$ is obtained from $\sigma$ by removing vertex $i$ are identity maps. Let $K\ \setminus \{i\}$ denote the full subcomplex $K_{\{1, \ldots, i-1, i+1, \ldots ,m\}}$. There is an equivalence
   \begin{equation*}
       \colim_{\sigma \in K}  \widehat{D}_I(\sigma)\simeq  \colim_{\sigma \in K\setminus\{i\}}  \widehat{D}(\sigma).
   \end{equation*}
   Iterating this process for each $i\not \in I$ yields, 
   \begin{equation*}
       \colim_{\sigma \in K}  \widehat{D}_I(\sigma)\simeq  \colim_{\sigma \in K\setminus\{i\}}  \widehat{D}(\sigma) \simeq \colim_{\sigma \in K_I}  \widehat{D}(\sigma) \simeq \widehat{(\underline{X},\underline{A})}^{K_I}. \qedhere
   \end{equation*}
		\end{proof}
		The previous proposition reduces the question about the stable homotopy type of a polyhedral product to understanding the homotopy type of $\widehat{(\underline{X},\underline{A})}^K$. 
  \begin{proposition}
  \label{prop:smashPolyhedralSplit}
      Let $K$ be a simplicial complex and consider a family of pairs $(\underline{X},\underline{A})$ where the map $\iota_i \colon A_i \to X_i$ is null for all $i$. Then there is an equivalence 
      \begin{equation*}
          \widehat{(\underline{X},\underline{A})}^K \simeq \bigvee_{\sigma \in K} \B{\calK_{<\sigma}}\star \widehat{D}(\sigma).
      \end{equation*}
  \end{proposition}
We will postpone the proof of Proposition \ref{prop:smashPolyhedralSplit} until Section \ref{sec:propproof}. We can now prove Theorem \ref{thm:stableSplitting} and Corollary \ref{cor:stableSplittingContractible}.
\begin{proof}[Proof of Theorem \ref{thm:stableSplitting}]
		The assertion follows by applying the result of Proposition \ref{prop:smashPolyhedralSplit} to the right-hand side of 
  \begin{equation*}
				\Sigma(\underline{X},\underline{A})^K \simeq \Sigma \left( \bigvee_{I\subset [m] } \widehat{(\underline{X},\underline{A})}^{K_I}\right).
			\end{equation*}
  from Proposition \ref{prop:stableSmashProdSplit}. 
\end{proof} 
\begin{proof}[Proof of Corollary \ref{cor:stableSplittingContractible}]
    Fix $I \subset [m]$. Since $X_i$ is contractible, it follows that $\widehat{D}(\sigma)$ is contractible for all $\sigma \in K$ where $\sigma \neq \emptyset$. Consequently, there is an equivalence
    \begin{equation*}
				\Sigma (\underline{X},\underline{A})^K \simeq \Sigma \left( \bigvee_{I\subset [m]}  \left( \B{(\calK_I)_{<\emptyset}}\star\widehat{D}(\emptyset)\right) \right) \simeq \Sigma \bigvee_{I \in K} \B{K_I}\star \widehat{A}^I. 
	\end{equation*}
    The space $\B{K_I}$ is contractible whenever $I \in K$, so we only need to consider $I \not \in K$.
\end{proof}
The rest of this section is dedicated to proving Proposition \ref{prop:smashPolyhedralSplit}.
\subsection{The proof of Proposition \ref{prop:smashPolyhedralSplit}}
\label{sec:propproof}
		We need to introduce some notation before we can get to the technical lemmas. 
  \begin{definition}
			Let $\calD$ be a poset category. A diagram $\frakX\colon \calD \to \calC$ is called a diagram with \emph{constant maps} if for all objects $a,b \in \calD$ and any nonidentity morphism~${f\colon a \to b, f \neq \text{id}_a}$ the map $\frakX(f)\colon \frakX(a)\to \frakX(b)$ is a constant map. In other words, the morphism $f$ can be factored as a composition of maps $\frakX(a) \to * \to \frakX(b)$. 
		\end{definition}
		\begin{definition}
			Let $\calD$ be a poset category. For an object $c \in \calC$ let $\calD_c$ be the diagram with the shape of $\calD$, but every object is mapped to the object $c$. Note that there is an equivalence  $\colim \calD_* \simeq \B{\calD}$ by Definition \ref{def:posetRealization}. Because $\calC$ is cartesian closed, there is an equivalence $\colim \calD_c \simeq c \times \B{\calD}$.
		\end{definition}

  The following is an $\infty$-categorical version of the \emph{initial diagram lemma} found in \cite[Lemma 3.4]{Welker1999}.
		\begin{lemma}
			\label{lem:initialdiagram}
			Let $\frakX$ be an initial diagram with constant maps over a poset category $\calD$ with initial object $a$. If $\frakX(b)=*$ for each $b\neq a$ then there is the following equivalence 
			\begin{equation*}
				\colim_\calD \frakX \simeq \frakX(a) \star \B{\calD_{<a}}. 
			\end{equation*}
			
		\end{lemma}
		\begin{proof}
			The diagram category $\calD$ is the category $\calD_{<a}$, but with an initial object. Let $\{a\}$ denote the single object category. We have an equivalence of categorties $\calD \simeq \{a\} \star \calD_{<a}$, where $\star$ denotes the join of categories as in \cite[§1.2.8]{LurieHTT}. To model the join, we have have following pushout of categories
			\begin{equation*}
				\{a\} \longleftarrow \calD_{<a}\times \{a\} \longrightarrow \calD_{<a}\times \Delta^1.
			\end{equation*} The diagram $\frakX$ is induced by the maps 
			\begin{equation*}
				\begin{tikzcd}
					\calD_{<a} \times \{a\} \arrow[d] \arrow[r] \arrow[dr, "\text{const}_{\mathfrak{X}(a)}" description] & \calD_{<a} \times \Delta^1 \arrow[d, "\text{const}_*" description] \\ \{a\} \arrow[r, "\mathfrak{X}(a)" description] 
					& \mathcal{C}.          
				\end{tikzcd}
			\end{equation*}
			The colimit of $\frak X$ is equivalent to the colimit of the diagram of diagrams
			\begin{equation*}
				\frakX(a) \longleftarrow (\calD_{<a})_{\frakX(a)} \longrightarrow  ((\calD_{<a}) \times \Delta^1)_*.
			\end{equation*}
            We have written it as a pushout of diagrams to ease notation. By~\cite[Proposition~4.4.2.2]{LurieHTT} we may first compute the colimits termwise in the pushout diagram. Thus, we are left we with a diagram
			\begin{equation*}
				\frakX(a) \longleftarrow \frakX(a) \times \B{\calD_{<a}} \longrightarrow \B{\calD_{<a}},\end{equation*} 
			which by definition has colimit equal to the join $\frakX(a) \star \B{\calD_{<a}}$.
		\end{proof}
        \begin{lemma}
			\label{lem:WedgeLemma}
			Suppose that $\frakX$ is a diagram where each morphism is null over an indexing poset category $\calD$ with initial object $a$. The colimit of the diagram $\frakX$ has the wedge decomposition
			\begin{equation*}
				\colim \frakX \simeq \bigvee_{a \in \mathrm{Obj}(\calD)}\left(\B{\calD_{<a}} \star \frakX(a)\right).
			\end{equation*}
		\end{lemma}
  \begin{proof}
        We will start by defining a couple of necessary diagrams. For each object $a \in \calD$, let~${X[a]  \colon \calD\to \calC}$ be the diagram such that ~$a \mapsto \frakX(a)$ and $b \mapsto *$ for each  $b\in \calD$ where~$b \not \simeq a$. Similarly, for each object~$a \in \calD$, let~$X[a]' \colon \calD_{\leq a}\to \calC$ be the diagram such that $a \mapsto \frakX(a)$ and ~${b  \mapsto *}$  for each $b \in \calD_{\leq a}$  where $b \not \simeq a$.  Since each morphism in the diagram $\frakX$ is null, it can be decomposed as a wedge of diagrams $X[a]$ for each object~$a \in \calD$. Thus, we have an equivalence 
        \begin{equation*}
            \colim \frakX \simeq \bigvee_{a\in \text{Obj}(\calD)} \colim X[a].
        \end{equation*}
        For each object $a \in  \calD$ the diagram $X[a]$ is a left Kan extension the functors 
        \begin{equation*}
            {X[a]'\colon \calD_{\leq a} \to \calC} \quad \text{and} \quad i\colon \calD_{\leq a} \to \calD.
        \end{equation*} Since left Kan extensions preserve colimits, there is an equivalence $\colim X[a] \simeq \colim X[a]'$.
        The diagram $X[a]'$ satisfies the conditions of Lemma~\ref{lem:initialdiagram}. Hence there is an equivalence 
        \begin{equation*}
            \colim X[a]' \simeq \B{\calD_{<a}}\star \frakX(a). \qedhere
        \end{equation*}   
  \end{proof}
		We can now prove Proposition \ref{prop:smashPolyhedralSplit}.

        \begin{proof}[Proof of Proposition \ref{prop:smashPolyhedralSplit}]
          
			Define the diagram $\widehat{E}\colon \calK \to \calC$ to be given by $\widehat{E}(\sigma) = \widehat{D}(\sigma)$ for all $\sigma \in K$, and for all $\sigma \subsetneq \tau$ the maps $\widehat{e}_{\sigma,\tau}\colon \widehat{E}(\sigma) \to \widehat{E}(\tau)$ to be the constant map to the basepoint. Since the maps $f_i\colon A_i \to X_i$ are null by assumption, the maps between~$\widehat{D}(\sigma)$ and $\widehat{D}(\tau)$ for $\sigma \subsetneq \tau$ will also be null-homotopic. We get the following equivalences 
			\begin{equation*}
				\widehat{(\underline{X},\underline{A})}^{K} \simeq \colim \widehat{D} \simeq \colim \widehat{E} \simeq \bigvee_{\sigma \in K} \B{\calK_{<\sigma}}\star \widehat{D}(\sigma)
			\end{equation*}
			since  $\widehat{E}$ satisfies the conditions of Lemma \ref{lem:WedgeLemma}.
		\end{proof}
		\section{Example categories}
  \label{sec:exampleCats}

In this section we discuss polyhedral products in several categories. In particular in Section \ref{sec:MotivicHomotopy}, we will introduce the category of motivic spaces and motivic moment-angle complexes, which will be our main focus for the rest of the paper. 
  \subsection{The category of spaces}
Let $\calS$ be the $\infty$-category of spaces. In Section \ref{sec:ClassicalPolyhedral} we defined polyhedral products in topological spaces as a union of topological spaces. As long as the map $A_i \to X_i$ is a cofibration, for each pair of spaces~$(X_i,A_i) \in (\underline{X},\underline{A})$, then by \cite[Lemma~3.1]{Welker1999} there is a homotopy equivalence
\begin{equation*}
  (\underline{X},\underline{A})^K :=  \colim_{\sigma \in K} D(\sigma) \simeq \bigcup_{\sigma \in K}D(\sigma).
\end{equation*}
It is important to emphasise that the colimit in $\calS$ is a higher categorical colimit, and thus corresponds to a homotopy colimit in classical homotopy theory. Thus, the results of Section \ref{sec:ClassicalPolyhedral} all follow from Section~\ref{sec:inftysetup} and \ref{sec:stableSplit} by letting $\calC = \calS$. 

		\subsection{The category of $G$-equivariant spaces}
        For a discrete group $G$, one can define the category $\calS^G$ of $G$-equivariant topological spaces. For a $G$-space $X$, denote the fixed points by~$X^G$. A map $f\colon X \to Y$  is a weak equivalence of~$G$-spaces if for every subgroup $H \leq G$ the restricted map $f^H \colon X^H \to Y^H$ is a weak equivalence of spaces. Due to Elmendorf's theorem \cite{Elmendorf}, the category of~$G$-spaces can be seen as a presheaf category, and is in particular an $\infty$-topos. This makes it possible to define $G$-equivariant polyhedral products coming from families of pairs~$(\underline{X},\underline{A})$ of~\hbox{$G$-spaces}. The stable splitting results from Section \ref{sec:stableSplit} also hold. 

        The moment-angle complex $\zk = (D^2,S^1)^K$ can be endowed with both a $C_2$-action (reflection) and a $S^1$-action (rotation). Similarly, the pair $\R\zk = (D^1,S^0)^K$ can be endowed with a $C_2$-action (reflection).  There is a relation between $\zk$ and $\R \zk$ through the fixed points of $\zk$ under the reflection.
        \begin{equation*}
            \zk^{C_2} = ((D^2,S^1)^K)^{C_2} \simeq ((D^2)^{C_2},(S^1)^{C_2})^K \simeq (D^1,S^0)^K = \R \zk.
        \end{equation*}
        When $K$ is a simplicial complex with $m$ vertices, there is an action of the torus ${T = (S^1)^{\times m}}$ on $\zka$ induced by the $S^1$-action on each pair $(D^2,S^1)$. For a freely acting subtorus $L$ of $T$, one can define the \emph{partial quotient} as the quotient $\zk / L$. The partial quotients are topological versions of smooth not necessarily projective toric varieties.
        
        Actions of a group $G$ on the simplicial complex $K$ can also product interesting equivariant examples. This approach does not require the pairs of spaces $(\underline{X},\underline{A})$ to be~$G$-spaces. However, this case does not allow for use of the results developed in Section \ref{sec:stableSplit}. Some of the results do still hold, but with modifications. Fu and Grbi\'c showed \cite[Theorem~3.3]{FuGrbic} that if $K$ is a simplicial complex with a $G$-action, then there is a homotopy $G$-equivalence 
        \begin{equation*}
            \Sigma^2 (X,A)^K \simeq \Sigma^2 \bigvee_{I \in K} \widehat{(X,A)}^K. 
        \end{equation*}
        This is a similar result as Proposition \ref{prop:smashPolyhedralSplit}, but with an extra suspension.
		\subsection{The category of motivic spaces}
  \label{sec:MotivicHomotopy}
  The category of motivic spaces over a base field $k$ was introduced by Morel and Voevodsky in \cite{MorelVoevodsky99}. Roughly speaking, the category of motivic spaces, also known as  the~$\AAA$-homotopy theory of $k$, is a homotopy theory for smooth schemes. Methods and concepts from algebraic topology had been in use in algebraic geometry for a long time before motivic homotopy theory, but the category of motivic spaces allowed for a framework where techniques from algebraic topology could systematically be lifted to algebraic geometry.  

  Let $k$ be a perfect field of characteristic different from $2$. Let $\smk$ be the category of smooth schemes of finite type over $k$. We denote by $\mathrm{PreSh}(\smk)$, the category of simplicial presheaves on $\smk$. We denote by $\Shnis(\smk)$, the category of simplicial Nisnevich sheaves on $\smk$. In motivic homotopy theory the affine line $\AAA$ takes on the role of the interval. We say that a presheaf $\calF\in \mathrm{PreSh}(\smk)$  is $\AAA$-invariant if there is an equivalence $\calF (X) \simeq \calF (X \times \AAA)$. The category of motivic spaces~$\Hk$ is the full subcategory of $\Shnis(\smk)$ spanned by $\AAA$-invariant Nisnevich sheaves $\calF$. In particular, there is a localization functor  $\Lmot\colon \mathrm{PrSh}(\smk) \to \Hk$, which is a left adjoint of the inclusion $\Hk \subset \mathrm{PrSh}(\smk)$. Small colimits are universal in the $\Hk$ \cite[Proposition 3.15]{Hoyois6ff}, which shows that it is cartesian closed. A more detailed explanation of the $\infty$-categorical construction of the category of motivic spaces can be found in \cite[Appendix C]{HoyoisTraceFormula} or \cite[§3]{Hoyois6ff}. A model categorical survey and introduction of unstable motivic homotopy theory can be found in \cite{Antieau2016APF}.
  
We will now survey some results about the motivic homotopy theory which can be found in Chapter 3 of \cite{MorelVoevodsky99}. The category of motivic spaces contains both geometric objects (schemes) and topological objects (simplicial sets). A scheme $X \in \smk$ can be seen as an element of $\Hk$ in the following way. For any $Y \in \smk$, the scheme~$X$ can be viewed as a simplicial presheaf by letting $X(Y) := \mathrm{Hom}_k(Y,X)$, the scheme morphisms of $Y$ to $X$. By motivic localization the presheaf represented by $X$ can be considered an object of $\Hk$. Let $S$ be a simplicial set. One can consider the constant simplicial presheaf $\mathrm{const}_S$. When it is clear that we are working with motivic spaces, we will abuse notation and write $X$ for $\Lmot\mathrm{Hom}_k({-},X) \in \Hk$ and~$S$ for the motivic localization~$\Lmot\mathrm{const}_S.$ 

   We will now look at a some motivic spaces. Recall that the affine line $\AAA$ plays the role of the interval in $\Hk$ and is contractible. That is, there is an equivalence~${\AAA \simeq \Spec(k)}$ in $\Hk$. A prominent feature of motivic homotopy theory is that the spheres are bigraded.  There is the simplicial circle, which is given by the constant simplicial presheaf to a simplicial set model of $S^1$. There is also a geometric circle, the punctured affine line $\Gm$, which is given by the scheme $\Gm := \AAA \setminus 0$. The motivic spheres are created by smashing copies of the geometric and simplicial circles. Let $S^{1,0}$ be the simplicial circle, and $S^{1,1}:=\Gm$. Thus for $a\geq b \geq 0$, we have $S^{a,b} = (S^1)^{\wedge (a-b)} \wedge \Gm^{\wedge b}$. Some of the higher dimensional spheres can be represented by schemes as well. A standard way of constructing the projective line $\PP^1$ in algebraic geometry is by gluing two affine lines along a common~$\Gm$. Thus $\PP^1$ is the colimit of the following diagram 
		\begin{equation*} \AAA \longleftarrow \Gm \longrightarrow \AAA. \end{equation*}Since $\AAA$ is contractible in $\calH(k)$, $\PP^1$ is equivalent to the colimit of the diagram  
		\begin{equation*} \Spec(k) \longleftarrow \Gm \longrightarrow \Spec(k), \end{equation*}
		which is $\Sigma \Gm \simeq S^{1,0} \wedge \Gm$. In other words, $\PP^1$ is equivalent to $S^{2,1}$. Higher dimensional punctured affine spaces are also models for motivic spheres. Let $n>0$, then~${\A^n \setminus 0 \simeq S^{2n-1,n}}$. 
		
  There are also ways of relating motivic homotopy theory to classical homotopy theory.  When $k$ is a subfield of $\C$ (resp.\ $\R$) and $X\in \smk$ , we will write $X(\C)$ (resp.\ $X(\R)$) for its complex (resp.\ real) points as a topological space. This yields a realization functor from $\calH(\C) \to \mathcal{S}$.  Furthermore, there is a second realization functor when $k$ is a subfield of $\R$. Whenever $X$ is a smooth scheme over $\R$, then $X(\C)$ is a $C_2$-space with the action of complex conjugation and we have the relation $X(\C)^{C_2} \simeq X(\R)$ as topological spaces. Thus we can describe the two realization functors as  
  \begin{equation*}
      \mathrm{Re}_\C \colon \calH(\C) \to \mathcal{S} \quad \text{and} \quad \mathrm{Re}_\R \colon \calH(\R) \to \mathcal{S}^{C_2}.
  \end{equation*} Let $S$ be a simplicial set, then both its complex and real realization of its associated object in $\calH(k)$ is the realization $S$ as a topological space, with trivial $C_2$-action under real realization. A reader not familiar with algebraic geometry might not understand why $\Gm$ could be an algebraic sphere at first glance. Consider the real number line, when we remove the origin we get a topological space which has the homotopy type of~$S^0$. Now consider the complex numbers, the space $\C\setminus 0$ has the homotopy type of $S^1$. So $\Gm(\C) \simeq S^1$ and $\Gm(\R) \simeq S^0$. For an arbitrary motivic sphere $S^{a,b}$, with $a\geq b \geq 0$,  we have $\mathrm{Re}_\C(S^{a,b}) \simeq S^a$ and $(\mathrm{Re}_\R(S^{a,b}))^{C_2} \simeq S^{a-b}$.
		
		As described in Section \ref{sec:MACs},  the moment-angle complex $(D^2,S^1)^K$ and the real moment-angle complex $(D^1,S^0)^K$ have both been extensively studied in the category of spaces. We will now introduce the \emph{motivic moment-angle complex}.
		\begin{definition}
			Let $K$ be a simplicial complex, we define the \emph{motivic moment-angle complex}~$\zka$ to be the polyhedral product 
			\begin{equation*}
				\zka := (\AAA,\Gm)^K
			\end{equation*}
   in the $\infty$-category $\Hk$.
		\end{definition}
        \begin{remark}
            When $k$ is a subfield of $\C$, using complex realization yields the equivalence~${\zka(\C) \simeq (\C,\C^\times)^K}$ which deformation retracts onto $\zk$ \cite[Theorem 4.7.5]{BuchPan}. Furthermore, if $k$ is a subfield of $\R$, there is a deformation retraction $\zka(\R) \simeq \R\zk$. 
        \end{remark}
        As noted earlier, the $\Hk$ is cartesian closed and has all small colimits. This makes it possible to apply Theorem \ref{thm:stableSplitting}. Since $\AAA$ is contractible, the following result an application of Corollary \ref{cor:stableSplittingContractible}.
		
		\begin{theorem}
			\label{thm:stableMMACsplit}
			Let $K$ be a simplicial complex. Then there is an equivalence in $\calH(k)$
			\begin{equation*}
				\Sigma \zka \simeq  \Sigma \left(\bigvee_{I \not \in K} \B{K_I}\star\Gm^{\wedge|I|}\right) \simeq \bigvee_{I \not \in K} \B{K_I}\wedge S^{|I|+2,|I|}.
			\end{equation*}
		\end{theorem}
		\begin{remark}
		    All of the results from Section \ref{sec:stableSplit} can be proven for $\calC = \calH(k)$ using Morel and Voevodsky's model structure and proving it locally on the value of simplicial presheaves using the results for topological spaces. 
		\end{remark}

\section{Affine models for motivic moment-angle complexes and toric varieties}
\label{sec:affine}
   In this section we will provide various models of $\zka$ in $\mathrm{Sm}_k$. We will also use affine models of $\zka$ to give affine models of smooth projective toric varieties. We begin with identifying $\zka$ with a smooth scheme.

\begin{proposition}
            \label{prop:monomialIdeal}
            Let $K$ be a simplicial complex, and let $\sigma_1, \ldots, \sigma_n$ be the maximal simplices of $K$. There is an equivalence 
            \begin{equation*}
                \zka \simeq \A^m \setminus L,
            \end{equation*} where $L$ is the variety cut out by the monomial ideal
            \begin{equation*}
               (\prod_{i\not \in \sigma_1} x_i, \ldots, \prod_{i\not \in \sigma_n} x_i) \subset k[x_1, \ldots, x_m].  
            \end{equation*}
        \end{proposition}
        \begin{proof}
            The colimit presentation of $\zka$ gives a Zariski cover of an algebraic variety. For each  simplex $\sigma \in K$, we have an equivalence $D(\sigma) = \A^m \setminus (\prod_{i \not \in \sigma} x_i)$. The algebraic variety $\zka$ can be described as the union
            \begin{equation*}
                \bigcup_{\sigma \in K} D(\sigma) = \bigcup_{\sigma \in K} \A^m \setminus (\prod_{i \not \in \sigma} x_i) = \A^m \setminus L.
            \end{equation*}
            We may enumerate the maximal simplices of $K$ as $\sigma_1, \ldots, \sigma_n$. The variety $L$ is cut out by the ideal 
            \begin{equation*}
               \bigcap_{1 \leq j \leq n}(\prod_{i\not \in \sigma_j} x_i) = (\prod_{i\not \in \sigma_1} x_i, \ldots, \prod_{i\not \in \sigma_n} x_i). \qedhere
            \end{equation*}
        \end{proof}
        \begin{remark}
        \label{rem:AlexanderDual}
            The ideal cutting out $L$ is the Stanley--Reisner ideal $I_{K^\vee}$ of Alexander dual $K^\vee$.
        \end{remark}   
        \begin{remark}
            From now on, whenever we speak about $\zka$ as a scheme, we will always mean $\A^m \setminus L$.
        \end{remark}
        The scheme $\A^m \setminus L$ can be identified with a toric variety generated by the following fan. Let $e_i$ denote the $i$th coordinate unit vector of $\R^m$. For each $\sigma \in K$, let~${C_\sigma:= \mathrm{Cone}(e_{i_1},\ldots, e_{i_n})}$ with $i_j \in \sigma$. The fan for $\zka$ is the collection of the cones~$C_\sigma$ for each $\sigma \in K$. Motivic moment-angle complexes have been studied before as toric varieties. In \cite{WendtFundamental} Wendt computes the $\AAA$-fundamental group of smooth toric varieties, this includes motivic-moment-angle complexes. 

        Our goal is to identify $\zka$ with an affine scheme. To do this we will need the following family of affine schemes. 
        \begin{definition}
             Let $L\subset \A^m$ be a closed subvariety cut out by the ideal $I = (f_1, \ldots, f_n)$. We define
            \begin{equation*}
                Q_I := \Spec\left(\frac{k[x_1,\ldots x_m,y_1,\ldots y_n]}{(f_1y_1 + \ldots +f_ny_n - 1)}\right)
            \end{equation*}
            and a morphism 
            \begin{equation*}
                \pi \colon Q_I \to \A^m \setminus L 
            \end{equation*}
            given by projection onto $(x_1, \ldots, x_m) \in  \A^m$.
        \end{definition}
        \begin{lemma}
        \label{lem:JouanolouComplement}
            Let $L\subset \A^m$ be a closed subvariety cut out by the ideal $I = (f_1, \ldots, f_n)$. The map $\pi \colon Q_I \to \A^m \setminus L $ is an $\AAA$-equivalence.
        \end{lemma}
        \begin{proof}
            The variety $\A^m \setminus L$ is covered by the opens $U_i$, where $f_i \neq 0$. Thus, locally for each $U_i$,  \begin{equation*}
                U_i \cong \Spec(k[x_1, \ldots, x_m][f_i^{-1}]).
            \end{equation*}
           Computing the preimage of $U_i$ yields
           \begin{equation*}
               \pi^{-1}(U_i) \cong  \Spec\left(\frac{k[x_1,\ldots x_m,y_1,\ldots y_n][f_i^{-1}]}{(f_1y_1 + \ldots +f_ny_n - 1)}\right).
           \end{equation*}
           We see that $\pi$ trivializes, and we get $\pi^{-1}(U_i) \cong U_i \times \A^{n-1}$. The opens $\pi^{-1}(U_i)$ cover~$Q_I$ as well since the ideal $I$ is a unit ideal in the coordinate ring of $Q_I$. The fibers of $\pi$ are trivial and the morphism is smooth, which implies that $Q_I$ is a Zariski locally trivial bundle over $\A^m \setminus L$ and $\pi$ is an $\AAA$-equivalence.  
        \end{proof}
         Proposition \ref{prop:monomialIdeal} shows that $\zka$ is homotopy equivalent to $\A^m\setminus L$ where~$L$ is some intersection of coordinate hyperplanes. As noted in Remark \ref{rem:AlexanderDual}, the variety~$L$ is cut out by a monomial ideal $I_{K^\vee}=(f_1, \ldots , f_n) \subset k[x_1, \ldots, x_m]$. 
         \begin{corollary}
         \label{cor:zkaaffinedual}
             Let $K$ be a simplicial complex. Then there is an $\AAA$-equivalence 
             \begin{equation*}
                 \zka \simeq Q_{I_{K^\vee}}.
             \end{equation*}
         \end{corollary}
         We will also give a different affine model of $\zka$ that uses the Stanley--Reisner ideal~$I_K=(f_1, \ldots, f_n)$. For simplicity, we will write $i\in f_j$ when $x_i$ is a factor of~$f_j$. 
         \begin{proposition}
         \label{prop:zkaaffine}
             Let $K$ be a simplicial complex on the vertex set $[m]$ and~${I_K=(f_1, \ldots, f_n)}$ its Stanley--Reisner ideal. There is an equivalence 
             \begin{equation*}
                \zka \simeq \Spec\left( \frac{k[x_1,\ldots,x_m,y_{ij}]}{(\sum_{i\in f_1}x_iy_{i1}-1,\ldots,\sum_{i\in f_n}x_iy_{il}-1)}\right).
             \end{equation*}
         \end{proposition}
         \begin{proof}
             One can cover $\A^m \setminus L$ by the affine opens corresponding to the maximal simplices of $K$, i.e.\ $D(\sigma)$ (as in Definition \ref{def:oopolyhedralprod}) where $\sigma$ is maximal in $K$. Using the same strategy as the proof as Lemma \ref{lem:JouanolouComplement} with this open cover yields the result.
         \end{proof}

         \begin{remark}
             For any $n>0$, the two affine representatives for $\zka$ from Corollary \ref{cor:zkaaffinedual} and Proposition \ref{prop:zkaaffine} are equal when $K = \partial \Delta^n$.
         \end{remark}
  In some special cases, it is possible to give an affine model of $\Sigma \zka$. In the special case where $K=\partial \Delta^{n-1}$ we have $Z_{\partial \Delta^{n-1}}^{\AAA}\simeq S^{2n-1,n}$, and there is an equivalence due to Asok, Doran, and Fasel \cite[Theorem 2.2.5]{AsokDoranFasel} 
           \begin{equation*}
               \Sigma Z_{\partial \Delta^{n-1}}^{\AAA} \simeq S^{2n,n} \simeq \Spec\left( \frac{k[x_1,\ldots, x_n, y_1, \ldots, y_n, z]}{(x_1y_1 + \ldots x_ny_n - z(1-z))}\right).
           \end{equation*}
However, we are able to say more about $\Sigma_{\PP^1}\zka$. By \cite[Corollary 4.16]{algebraicSuspension}, the motivic space~$\Sigma_{\PP^1}\zka$ admits an affine model because $\zka$ has the homotopy type of an affine scheme. The following remark shows how the $\PP^1$-suspension of $\zka$ is a motivic moment-angle complex.
\begin{remark}
    Let $K$ and $\Delta^{m-1}$ be a simplicial complex on the vertex set $[m]$. We define~$M=K \star \{m+1\}$ on the vertex set $[m+1]$ Consider the pushout of simplicial complexes 
    \begin{equation*}
        \begin{tikzcd}
        \overline{K} \arrow[dr, phantom, "\scalebox{1}{$\ulcorner$}" , very near end, color=black]\ar[r]\ar[d] & \overline{M}  \ar[d] \\
    \overline{\Delta^{m-1}} \ar[r] & K'.
    \end{tikzcd}
    \end{equation*}
    By Proposition \ref{prop:PushoutPolyProd}, this induces a pushout of motivic moment-angle complexes 
    \begin{equation*}
        \begin{tikzcd}
        Z_{\overline{K}}^{\AAA} \arrow[dr, phantom, "\scalebox{1}{$\ulcorner$}" , very near end, color=black]\ar[r]\ar[d] & Z_{\overline{M}}^{\AAA}  \ar[d] \\
        Z_{\overline{\Delta^{m-1}}}^{\AAA} \ar[r] & Z_{K'}^{\AAA}.
    \end{tikzcd}
    \end{equation*}
    We can express the motivic moment-angle complexes in the diagram in terms of $\zka$, and this gives
    \begin{equation*}
        \begin{tikzcd}
        Z_{K}^{\AAA}\times \Gm \arrow[dr, phantom, "\scalebox{1}{$\ulcorner$}" , very near end, color=black]\ar[r]\ar[d] & Z_{K}^{\AAA} \times \AAA \ar[d] \\
        \A^m \times \Gm \ar[r] & Z_{K'}^{\AAA}.
    \end{tikzcd}
    \end{equation*}
    Since $\AAA$ is contractible, we get an equivalence
    \begin{equation*}
        Z_{K'}^{\AAA} \simeq \zka \star \Gm \simeq \Sigma_{\PP^1}\zka.
    \end{equation*}
    Thus $\Sigma_{\PP^1}\zka$ is a motivic moment-angle complex and has the homotopy type of a smooth affine scheme. 
\end{remark}
	Recall that any smooth projective toric variety $X$ can be realized as a quotient of a motivic-moment-angle complex under the action of a torus \cite{CoxHomogCoord}. The torus action induces an action on $Q_K$. By computing the ring of invariants of the coordinate ring of $Q_K$ under the torus action, we can find an explicit smooth affine description of  $X$. Concretely, if~$X = \zka / T$ is a smooth projective toric variety for some simplicial complex~$K$ and torus~$T$, then $Q_K / T$ is an affine bundle over $X$ with trivial fibers. 

\begin{example}
    \label{ex:p1affine} Let $K$ be two disjoint points, then $\zka \simeq \A^2 \setminus 0$. We make $\Gm$ act on~$\A^2 \setminus 0$ by scalar multiplication. That is, for $\lambda \in \Gm$, we define $\lambda \cdot (x_1,x_2) = (\lambda x_1, \lambda x_2)$. By Corollary \ref{cor:zkaaffinedual}, there is an equivalence 
    \begin{equation*}
        \zka \simeq \A^2 \setminus 0 \simeq \Spec\left( \frac{k[x_1,x_2,f_1,f_2]}{(x_1f_1+x_2f_2-1)}\right) = \mathrm{SL}_2.
    \end{equation*}
    The action of $\Gm$ on $\A^2 \setminus 0$ extends to an action on $\mathrm{SL}_2$ as follows 
    \begin{equation*}
        \lambda \cdot (x_1,x_2,f_1,f_2) = (\lambda x_1, \lambda x_2, \frac{f_1}{\lambda}, \frac{f_2}{\lambda}).
    \end{equation*}
    Computing the ring of invariants of yields the generators $x_1f_1,x_2f_2,x_1f_2$, and $x_2f_1$. There are also relations $x_1f_1+x_2f_2=1$ and $x_1f_1 \cdot x_2f_2 = x_1f_2 \cdot x_2f_1$. Thus the ring of invariants is isomorphic to the ring 
    \begin{equation*}
        R = \frac{k[a,b,c,d]}{(a+d-1,ad-bc)}.
    \end{equation*}
    Thus $\PP^1 \simeq \Spec(R)$. One can view $R$ as a $(2\times 2)$-matrix with trace $1$ and determinant~$0$. This is equivalent to a rank $1$ idempotent matrix. Similar techniques can be applied to identify $\PP^{n-1}$ with an idempotent $(n \times n)$ matrices of rank $1$.
\end{example}
Let $K$ be a square, then $\zka \simeq \A^2 \setminus 0 \times \A^2 \setminus 0$. We can consider a $2$-dimensional torus acting on $\zka$ in the following way. Let $\lambda = (\lambda_1, \lambda_2) \in \Gm^{\times 2}$, we let $\Gm^{\times 2}$ act on~$\zka$ in the following way 
      \begin{equation*}
          \lambda \cdot (x_1,x_2,x_3,x_4) = (\lambda_1 x_1, \lambda_1 x_2, \lambda_2 x_3, \lambda_2 x_4). 
      \end{equation*}
      The GIT quotient of the action is $\PP^1 \times \PP^1$. The following two examples show how the two different affine models for $\zka$ yield different affine models for a toric variety. 
		\begin{example}
		By Corollary \ref{cor:zkaaffinedual}, there is an equivalence $\zka \simeq Q_{I_{K^\vee}}$. In the case where $K$ is a square, this yields 
  \begin{equation*}
      \zka \simeq \Spec\left(\frac{k[x_1,x_2,x_3,x_4,f_{13},f_{14},f_{23},f_{24}]}{(x_1x_3f_{13}+ x_1x_4f_{14} + x_2x_3f_{23} + x_2x_4f_{24}-1)} \right).
  \end{equation*}
  The action of $\Gm^{\times 2}$ on $\zka$ by $\Gm^{\times 2}$ as previously described extends to an action on $Q_{I_{K^\vee}}$ as follows
  \begin{equation*}
          \lambda \cdot (x_1,x_2,x_3,x_4,f_{13},f_{14},f_{23},f_{24}) = (\lambda_1 x_1, \lambda_1 x_2, \lambda_2 x_3, \lambda_2 x_4,\frac{f_{13}}{\lambda_1\lambda_2},\frac{f_{14}}{\lambda_1\lambda_2},\frac{f_{23}}{\lambda_1\lambda_2},\frac{f_{24}}{\lambda_1\lambda_2}). 
      \end{equation*}
     The $16$ generators of the ring of invariants under the action of the torus are $x_ix_jf_{pq}$ for~$i=1,2$, $j=3,4$, $p=1,2$, and $q=3,4$. 
 The ring of invariants has the relation \begin{equation*}
     x_1x_3f_{13}+ x_1x_4f_{14} + x_2x_3f_{23} + x_2x_4f_{24}=1
 \end{equation*} as well as \begin{equation*}
     x_ix_jf_{pq}\cdot x_{i'}x_{j'}f_{p'q'}= x_{i}x_{j}f_{p'q'} \cdot x_{i'}x_{j'}f_{pq}.
 \end{equation*}
 The variety $Q_{I_{K^\vee}}/T$ can be seen as the pullback of the Segre embedding of $\PP^1  \times \PP^1$ in~$\PP^3$ and the morphism $\pi \colon \Tilde{\PP}^3 \to \PP^1$, where $\Tilde{\PP}^3$ is the affine replacement for $\PP^3$.

		\end{example}
  
 	\begin{example}
 	    When $K$ is a square, there is an equivalence by Proposition \ref{prop:zkaaffine} 
      \begin{equation*}
            \zka \simeq \Spec\left( \frac{k[x_1,x_2,x_3,x_4,f_1,f_2,f_3,f_4]}{(x_1f_1+x_2f_2-1, x_3f_3+x_4f_4-1)}\right).
      \end{equation*}   
      In this case, the action of $\Gm^{\times 2}$ on $\zka$ by $\Gm^{\times 2}$ extends as follows
      \begin{equation*}
          \lambda \cdot (x_1,x_2,x_3,x_4,f_{13},f_{14},f_{23},f_{24}) = (\lambda_1 x_1, \lambda_1 x_2, \lambda_2 x_3, \lambda_2 x_4,\frac{f_{1}}{\lambda_1},\frac{f_{2}}{\lambda_1},\frac{f_{3}}{\lambda_2},\frac{f_{4}}{\lambda_2}). 
      \end{equation*}
      The eight generators of the ring of invariants are
      \begin{equation*}
          x_if_j, x_pf_q \quad \text{for} \quad \text{$i=1,2$, $j=1,2$, $p=3,4$, $q=3,4.$}
      \end{equation*}We have relations $x_1f_1+x_2f_2=1$ and $x_3f_3+x_4f_4=1$, as well as~${x_if_i\cdot x_jf_j = x_if_j \cdot x_jf_i}$ for $(i,j)=(1,2)$ and $(i,j) = (3,4)$. The ring of invariants is isomorphic to the tensor product of two copies of the coordinate ring of the affine replacement of $\PP^1$ from Example~\ref{ex:p1affine}.
 	\end{example}		
		\section{Invariants of motivic polyhedral products}
\label{sec:motivicInvariants}
		In this section we will consider various invariants for objects in the motivic homotopy category and apply them to motivic moment-angle complexes.
		
		\subsection{Cellular $\A^1$-homology}
 
 		In \cite{morel2023cellular} Morel and Sawant define cellular $\AAA$-homology for cellular varieties. Cellular varieties are smooth schemes that admits a nice stratification. The cellular $\AAA$-homology takes values in the derived category of strictly $\AAA$-invariant Nisnevich sheaves of abelian groups on $\smk$, which we will denote by $D(Ab_{\mathbb{A}^1}(k))$. Morel and Sawant define a scheme~$X$ to be \emph{cohomologically trivial} if $\mathbf{H}_n^{\mathrm{Nis}}(X,\mathbf{M})=0$, for every $n \geq 1$ and~$\mathbf{M}\in Ab_{\AAA}(k)$ \cite[Definition 2.9]{morel2023cellular}. Examples of cohomologically trivial schemes are~$\AAA, \Gm$, and products of cohomologically trivial schemes \cite[Remark 2.10]{morel2023cellular}. 
  \begin{definition}[{\cite[Definition 2.11]{morel2023cellular}}]
      Let $X \in \smk$ be a smooth $k$-scheme. A \emph{cellular structure} on $X$ consists of an increasing filtration 
      \begin{equation*}
          \emptyset = \Omega_{-1} \subsetneq \Omega_0 \subsetneq \Omega_1 \subsetneq  \ldots \subsetneq \Omega_n = X
      \end{equation*}
      by open subschemes of $X$ such that for each $0 \leq i \leq n$, the reduced induced subscheme~$X_i := \Omega_i \setminus \Omega_{i-1}$ of $\Omega_i$ is smooth, affine, everywhere of codimension $i$, and cohomologically trivial. We say that $X$ is a \emph{cellular scheme} if $X$ admits a cellular structure.
  \end{definition}
  This definition is meant to imitate the $CW$-structure of a topological space. See \cite[Remark 2.12.(2)]{morel2023cellular} for further details. 
  \begin{proposition}
  \label{prop:cellularStructure}
      Let $K$ be a simplicial complex, and let $K_i$ denote the $i$-skeleton of $K$. Let $s$ be the smallest integer such that $K_s = K$. Then \begin{equation*}
          \emptyset \subset \Gm^{\times m} \subsetneq Z_{K_0}^{\AAA} \subsetneq  Z_{K_1}^{\AAA} \subsetneq  \ldots \subsetneq Z_{K_s}^{\AAA} = \zka
      \end{equation*}
      is a cellular structure on $\zka$.
  \end{proposition}
  \begin{proof}
  By Proposition \ref{prop:monomialIdeal}, for each $0\leq i \leq s$ the variety $Z^{\AAA}_{K_i}$ is the open complement of the variety $V(I_{K^\vee_i})$ cut out by a monomial ideal $I_{K^\vee_i}$ in $\A^m$. Thus,
  \begin{equation*}
       X_{i+1} = Z^{\AAA}_{K_i} \setminus Z^{\AAA}_{K_{i-1}} = V(I_{K^\vee_i}) \setminus V(I_{K^\vee_{i-1}}) = \bigsqcup_{\sigma \in K, |\sigma|=i}\Gm^{\times (m-i)}.
  \end{equation*}
 Since $\Gm$ is cohomologically trivial, $\zka$ admits a cellular structure.
  \end{proof}
  Using the cellular structure of cellular varieties, Morel and Sawant define an $\AAA$-chain complex. From the cellular structure on $\zka$, one can create a cellular $\AAA$-chain complex in the fashion of Morel and Sawant. However, if we want to exploit the homotopical properties of polyhedral products, we are going to need a different cellular $\AAA$-chain complex. This is no problem since \cite[Corollary 2.42]{morel2023cellular} shows that any two cellular~{$\AAA$-chain} complexes of a $\zka$ will be homotopy equivalent in $D(Ab_{\AAA}(k))$. We will now show that the functor $C_*^{\mathrm{cell}}({-})$ sends motivic moment-angle complexes to a polyhedral product in $D(Ab_{\AAA}(k))$.

  \begin{proposition}
            The functor $C_*^{\mathrm{cell}}({-})$ preserves colimits of cohomologically trivial objects.
        \end{proposition}
        \begin{proof}
            The functor $C_*^{\mathrm{cell}}({-})$ is a pro-left adjoint to the category $\text{pro-}D(Ab_{\AAA}(k))$. When the objects are cohomologically trivial, their image in $\text{pro-}D(Ab_{\AAA}(k))$ are constant, and can be represented as elements of $D(Ab_{\AAA}(k))$. See \cite[Corollary 2.38 and Remark 2.39]{morel2023cellular} for further details.
        \end{proof}
        By \cite[Lemma 2.31]{morel2023cellular}, for $X,Y$ smooth cellular schemes, there is an isomorphism of chain complexes 
  \begin{equation*}
      C_*^{\mathrm{cell}}(X \times Y) \cong C^{\mathrm{cell}}_*(X)\otimes C_*^{\mathrm{cell}}(Y).
  \end{equation*}
We say that a chain complex $C_*$ of strictly $\AAA$-invariant sheaves is \emph{pointed}, if $C_0$ admits~$\Z$ as a direct summand. We denote the reduced chain complex of $C$ by $\widetilde{C}$ and there are isomorphisms $C_0 \cong \Z \oplus \widetilde{C}_0$ and $C_i \cong \widetilde{C}_i$ for $i>0$.    In the case where $C_*$ and $D_*$ are two pointed chain complexes concentrated in degree $0$, we get the following splitting 
  \begin{equation*}
      C_*\otimes D_* \simeq \Z \oplus \widetilde {C}_0 \oplus \widetilde{D}_0 \oplus  (\widetilde {C}_0 \otimes \widetilde{D}_0).
  \end{equation*}
   When $X$ is a pointed space that admits a cellular structure then $C^{\mathrm{cell}}_*(X)$ is a pointed chain complex. Because motivic moment-angle complexes are built out of products of $\AAA$'s and $\Gm$'s, it is important to understand the cellular structure of those pieces. Denote the $n$-th unramified Milnor--Witt $K$-theory sheaf by $\kmw_n$ (see \cite[§3.2]{Morel1} for a definition). Going forward, we will need the following property of the Milnor--Witt~{$K$-theory} sheaves. For $i,j \geq 0$ there is an isomorphism $\kmw_i \otimes \kmw_j \cong \kmw_{i+j}$. The following result is due to Morel and Sawant. 
		
		\begin{proposition}
			The cellular $\AAA$-chain complex for $\AAA$ and $\Gm$ are given by 
			\begin{equation*}
				C_i^{\mathrm{cell}}(\AAA) = \begin{cases}
					\Z & i=0, \\ 
					0 & i>0, 
				\end{cases} \quad \text{and} \quad  C_i^{\mathrm{cell}}(\Gm) = \begin{cases}
					\Z\oplus\kmw_1 & i=0, \\ 
					0 & i>0. 
				\end{cases}
			\end{equation*}
		\end{proposition}
		 When $S\in Ab_{\AAA}(k)$, we will abuse notation and write $S$ for the chain complex concentrated in degree $0$ with value $S$. Since $C^{\mathrm{cell}}({-})$ preserves products and colimits, it sends the motivic moment-angle complex to a polyhedral product in $ D_{\AAA}(Ab(k))$. We can now define a second $\AAA$-chain complex for $\zka$ as a polyhedral product
         \begin{equation*}
             C^{\mathrm{cell}}_*(\zka) \simeq  (C_*^{\mathrm{cell}}(\AAA),C_*^{\mathrm{cell}}(\Gm))^K=(\Z,\Z \oplus \kmw_1)^K.
         \end{equation*} 
        Recall that for a simplical complex $K$, we denote the geometric realization of $K$ by $\B{K}$ as in Definition \ref{def:simpcplxRealization}. In $D_{\AAA}(Ab(k))$, the geometric realization $\B{K}$ is represented by any singular chain complex that computes the singular homology of $K$ as a topological space. 
         \begin{proposition}
         \label{prop:zkachaincomplex}
             There the following is an equivalence of of chain complexes in $D(Ab_{\AAA}(k))$
             \begin{equation*}
             (C_*^{\mathrm{cell}}(\AAA),C_*^{\mathrm{cell}}(\Gm))^K \simeq  \bigvee_{I\not \in K} \Sigma \B{K_I}\wedge (\Z \oplus \kmw_{|I|}).
        \end{equation*}
         \end{proposition}
         \begin{proof}
             We saw earlier how tensor products of pointed complexes concentrated in degree zero splits into a wedge of complexes. Since $(C_*^{\mathrm{cell}}(\AAA),C_*^{\mathrm{cell}}(\Gm))^K$ is a colimit of tensor products of complexes concentrated in degree $0$, we get the following chain of equivalences
            \begin{equation*}
                (C_*^{\mathrm{cell}}(\AAA),C_*^{\mathrm{cell}}(\Gm))^K \simeq \colim_{\sigma \in K} D(\sigma) \simeq \colim_{\sigma \in K} (\Z \oplus \kmw_1)^{\otimes (m-|\sigma|)}.
            \end{equation*}
            We can now split the diagram into a wedge of diagrams. For each $I\subseteq [m]$ and $\sigma \in K$ define
            \begin{equation*}
                E_I(\sigma) = \begin{cases}
                    \Z \oplus \kmw_{|I|} &  \sigma \not \subseteq I, \\
                    \Z &  \sigma \subseteq I.
                \end{cases}
            \end{equation*}
            There is an equivalence
            \begin{equation*}
                \colim_{\sigma \in K} (\Z \oplus \kmw_1)^{\otimes (m-|\sigma|)}  \simeq \bigvee_{I\subseteq [m]} \colim_{\sigma \in K} E_I(\sigma).
            \end{equation*}
             We are now in a similar situation as in the proof of Proposition \ref{prop:smashPolyhedralSplit}. Fix $I \subset [m]$ and look at the colimit 
             \begin{equation*}
                 \colim_{\sigma \in K} E_I(\sigma).
             \end{equation*}
             For each $i \not \in I$, the maps induced by $\tau \subset \sigma$ obtained by removing vertex $i$ are identity maps. We get an equivalence of colimits 
             \begin{equation*}
                 \colim_{\sigma \in K} E_I(\sigma) \simeq \colim_{\sigma \in K_I} E'_I(\sigma),
             \end{equation*}
             where $E'_I(\sigma) = \Z \oplus \kmw_{|I|}$ if $\sigma = \emptyset$ and $E'_I(\sigma) = \Z$ otherwise. This is a diagram that satisfies Lemma \ref{lem:initialdiagram}, and we get an equivalence 
             \begin{equation*}
                 \colim_{\sigma \in K} E_I(\sigma) \simeq \Sigma \B{K_I} \wedge (\Z \oplus \kmw_{\B{I}}).
             \end{equation*}
            With the colimit of each wedge summand computed, we get
             \begin{equation*}
                 (C_*^{\mathrm{cell}}(\AAA),C_*^{\mathrm{cell}}(\Gm))^K  \simeq \bigvee_{I\subseteq [m]} \colim_{\sigma \in K} E_I(\sigma) \simeq \bigvee_{I\subseteq [m]} \Sigma \B{K_I} \wedge (\Z \oplus \kmw_{\B{I}}).
             \end{equation*}
             Since $\B{K_I}$ is contractible if $I \in K$, we only need to consider $K \not \in I$. This proves the claim.
         \end{proof}
         
        This splitting result can be seen as an unstable version of Theorem \ref{thm:stableSplitting}. Computing the homology of the chain complex from Proposition \ref{prop:zkachaincomplex} yields the following.
        \begin{corollary}
        \label{cor:homology}
            Let $K$ be a simplicial complex. There is an isomorphism 
            \begin{equation*}
        \widetilde{\mathbf{H}}_*^{\mathrm{cell}}(\zka) \cong \bigoplus_{I \not \in K} \widetilde{\mathbf{H}}_*(\Sigma \B{K_I}\wedge (\Z \oplus \kmw_{\B{I}}).
        \end{equation*}
        \end{corollary}
            We wish to describe the cellular $\AAA$-homology of $\zka$ in terms of $K$, so the next natural step is to understand what smashing with $(\Z \oplus \kmw_{|I|})$ and suspending $\B{K_I}$ does to the homology. 
		\begin{definition}
			For a chain complex $C$ with differential $d$, the cone of $C$ is defined  as the chain complex $\textrm{Cone}(C)_n := C_n \oplus C_{n-1} \oplus C_n$, with differential 
            \begin{equation*}
                \delta_n(a,b,c) = (da + b, -db, dc-b).
            \end{equation*}
		\end{definition}
		For a diagram $( Y \longleftarrow X \longrightarrow Z )$ of chain complexes $X$, $Y$, and $Z$, with chain maps~$i\colon X \to Y$ and $j\colon X \to Z$, we can create the following chain complex modeling the homotopy pushout in a category of chain complexes \begin{equation*}
		    C_n = Y_n \oplus X_{n-1} \oplus X_{n} \oplus X_{n-1} \oplus Z_n
		\end{equation*} with differential
		\begin{equation*}
			\partial(y,x_{n-1},x_n,x_{n-1}',z_n)= \left(dy+i(x_{n-1}),-dx_{n-1},dx_n - x_{n-1}+x_{n-1}', -dx_{n-1}',dz+j(x_{n-1}')\right).
		\end{equation*}
        A visualization of the complex can be seen below. The complex $C$ is the homotopy pushout ${( Y \longleftarrow X \longrightarrow Z )}$.
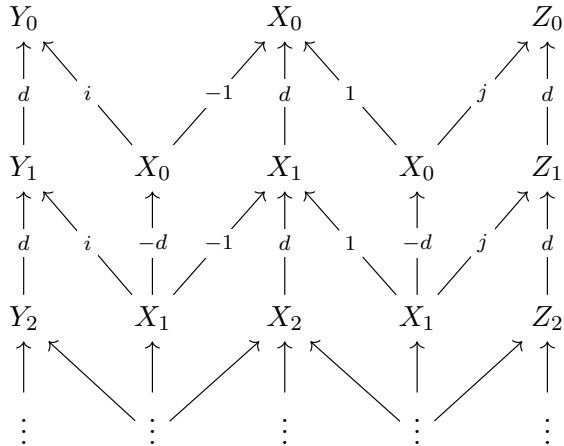
		\[\begin{tikzcd}
			{Y_0} && {X_0} && {Z_0} \\
			\\
			{Y_1} & {X_0} & {X_1} & {X_0} & {Z_1} \\
			\\
			{Y_2} & {X_1} & {X_2} & {X_1} & {Z_2} \\
			\vdots & \vdots & \vdots & \vdots & \vdots
			\arrow["d"{description}, from=3-3, to=1-3]
			\arrow["d"{description}, from=5-3, to=3-3]
			\arrow["{-1}"{description}, from=3-2, to=1-3]
			\arrow["1"{description}, from=3-4, to=1-3]
			\arrow["{-1}"{description}, from=5-2, to=3-3]
			\arrow["1"{description}, from=5-4, to=3-3]
			\arrow["d"{description}, from=3-1, to=1-1]
			\arrow["d"{description}, from=5-1, to=3-1]
			\arrow["d"{description}, from=5-5, to=3-5]
			\arrow["d"{description}, from=3-5, to=1-5]
			\arrow["i"{description}, from=3-2, to=1-1]
			\arrow["j"{description}, from=3-4, to=1-5]
			\arrow["j"{description}, from=5-4, to=3-5]
			\arrow["i"{description}, from=5-2, to=3-1]
			\arrow["{-d}"{description}, from=5-2, to=3-2]
			\arrow["{-d}"{description}, from=5-4, to=3-4]
			\arrow[from=6-1, to=5-1]
			\arrow[from=6-2, to=5-2]
			\arrow[from=6-3, to=5-3]
			\arrow[from=6-4, to=5-4]
			\arrow[from=6-5, to=5-5]
			\arrow[from=6-4, to=5-5]
			\arrow[from=6-4, to=5-3]
			\arrow[from=6-2, to=5-3]
			\arrow[from=6-2, to=5-1]
		\end{tikzcd}\]

  In the category of pointed chain complexes of strictly $\AAA$-invariant sheaves, the constant sheaf $\Z$ is the terminal object. We can now compute the wedge of two pointed complexes. 
		\begin{proposition}
		    Let $C,D \in \calD(Ab_{\AAA}(k))$ be pointed chain complexes with $C_0 = \Z \oplus \widetilde{C}_0$ and $D_0 = \Z \oplus \widetilde{D}_0$. Then the wedge $C\vee D$ is the chain complex
      \begin{equation*}
          (C\vee D)_n = \begin{cases}
              \Z \oplus \widetilde{C}_0 \oplus \widetilde{D}_0 & n = 0, \\
              C_n \oplus D_n & n > 0.
          \end{cases}
      \end{equation*}
      \begin{proof}
          Compute the homotopy pushout as in Definition \ref{def:wedge} using a cone on the chain complex of the point. 
      \end{proof}
		\end{proposition}
  We can now compute the smash product of $C\wedge D \in \calD(Ab_{\AAA}(k))$  as the homotopy pushout of the square 
  \begin{equation*}
      * \leftarrow C \vee D \rightarrow C\otimes D.
  \end{equation*}

		\begin{lemma}
   \label{lem:homologyGMsmash}
            Let $C$ be a pointed chain complex of strictly $\AAA$-invariant sheaves and let~$n>0$, then there is an isomorphism
            \begin{equation*}
                \mathbf{H}_i(C \wedge (\Z \oplus \kmw_n)) \cong\begin{cases}
                \Z \oplus (\widetilde{\mathbf{H}}_0(C)\otimes \kmw_n) & i=0,\\
                    \mathbf{H}_i(C) \otimes \kmw_n  & i > 0.
                \end{cases}
            \end{equation*}
            
		\end{lemma}
		\begin{proof}
			To  prove the quasi-isomorphism, we will compute the homology. We have the equivalence  
			\begin{equation*}
				C \wedge (\Z \oplus \kmw_n) \simeq \colim(* \longleftarrow  C \vee (\Z \oplus \kmw_n) \longrightarrow C \otimes (\Z \oplus \kmw_n).
			\end{equation*}
			
			  The chain complex for $C \wedge (\Z \oplus \kmw_n)$ can be modeled as the following homotopy pushout:
\[\begin{tikzcd}[cramped,column sep=small]
	\Z & {\widetilde{C}_0\oplus\Z\oplus \kmw_n} && {\widetilde{C}_0\otimes \kmw_n\oplus \widetilde{C}_0\oplus \kmw_n \oplus \Z} \\
	{\widetilde{C}_0\oplus\Z\oplus \kmw_n} & {C_1} & {\widetilde{C}_0\oplus\Z\oplus \kmw_n} & {C_1\otimes \kmw_n\oplus C_1} \\
	{C_1} & {C_2} & {C_1} & {C_2\otimes \kmw_n\oplus C_2} \\
	{C_2} & {C_3} & {C_2} & {C_3\otimes \kmw_n\oplus C_3} \\
	\vdots & \vdots & \vdots & \vdots
	\arrow["{-1}"{description}, from=2-1, to=1-2]
	\arrow["\epsilon", from=2-1, to=1-1]
	\arrow["1"{description}, from=2-3, to=1-2]
	\arrow["{-1}"{description}, from=3-1, to=2-2]
	\arrow["{-1}"{description}, from=4-1, to=3-2]
	\arrow["1"{description}, from=4-3, to=3-2]
	\arrow["1"{description}, from=3-3, to=2-2]
	\arrow["i"{description}, from=2-3, to=1-4]
	\arrow["i"{description}, from=3-3, to=2-4]
	\arrow["i"{description}, from=4-3, to=3-4]
	\arrow[from=2-4, to=1-4]
	\arrow[from=4-4, to=3-4]
	\arrow[from=3-4, to=2-4]
	\arrow[from=4-3, to=3-3]
	\arrow[from=3-3, to=2-3]
	\arrow[from=4-2, to=3-2]
	\arrow[from=3-2, to=2-2]
	\arrow[from=2-2, to=1-2]
	\arrow[from=4-1, to=3-1]
	\arrow[from=3-1, to=2-1]
	\arrow[from=5-4, to=4-4]
	\arrow[from=5-3, to=4-3]
	\arrow[from=5-3, to=4-4]
	\arrow[from=5-3, to=4-2]
	\arrow[from=5-2, to=4-2]
	\arrow[from=5-1, to=4-1]
	\arrow[from=5-1, to=4-2]
\end{tikzcd}\]
			All vertical arrows except $\epsilon$ are inherited from the differential $d$ on $C$. The map $\epsilon$ is projection onto $\Z$. The map $i$ is the inclusion. We may restrict $\partial_1$ to the factors~$\Z, \kmw_n$, and $\widetilde{C}_0$.
            \begin{equation*}
			   \partial_1^{\Z} =  \begin{pmatrix}
			        \mathrm{Id}_{\Z} & -\mathrm{Id}_{\Z} & 0 \\ 0 & \mathrm{Id}_{\Z} & \mathrm{Id}_{\Z}
			    \end{pmatrix} \quad \quad  \partial_1^{\mathrm{MW}}=
                 \begin{pmatrix}
			        -\mathrm{Id}_{\kmw_n} &0 \\ \mathrm{Id}_{\kmw_n} & -\mathrm{Id}_{\kmw_n}
			    \end{pmatrix}
                \quad \quad  \partial_1^{\widetilde{C}_0}=
                 \begin{pmatrix}
			        -\mathrm{Id}_{\widetilde{C}_0} &0 \\ \mathrm{Id}_{\widetilde{C}_0} & -\mathrm{Id}_{\widetilde{C}_0}
			    \end{pmatrix}
			\end{equation*}
            It is straightforward to check that all of the three restricted differentials above are surjective. The map $d_1^{\mathrm{MW}}\colon C_1 \otimes \kmw_n \to \widetilde{C}_0 \otimes \kmw_n$ has image $\Img(d_1) \otimes \kmw_n$. Thus \begin{equation*}
                \mathbf{H}_0(C \wedge (\Z \oplus \kmw_n)) \cong \Z \oplus (\widetilde{\mathbf{H}}_0(C)\otimes \kmw_n).
            \end{equation*}
            
            To compute $\ker \partial_1$, we first note that $\ker \partial_1^{\Z}$,   $\ker  \partial_1^{\mathrm{MW}}$, and $\ker \partial_1^{\widetilde{C}_0}$ are all trivial. We get $\ker d_1^{\mathrm{MW}} = \ker(d_1) \otimes \kmw_n$. The map onto $\widetilde{C}_0 \oplus \widetilde{C}_0$ is given by the matrix
            \begin{equation*}
                \partial_1^C = \begin{pmatrix}
			        -\mathrm{Id}_{\widetilde{C}_0} & 0\\ \mathrm{Id}_{\widetilde{C}_0} & \mathrm{Id}_{\widetilde{C}_0} \\ d_1 & 0 \\ 0 & d_1
			    \end{pmatrix} \colon \widetilde{C}_0 \oplus \widetilde{C}_0 \oplus C_1 \oplus C_1 \rightarrow \widetilde{C}_0 \oplus \widetilde{C}_0.
            \end{equation*}
            Since $\Img d_1 \subset \widetilde{C}_0$, the kernel is $\ker \partial_1^C = C_1 \oplus C_1$. Thus $\ker \partial_1 = \ker \partial_1^C \oplus \ker d_1^{\mathrm{MW}}$. This result also extends to $\ker \partial_i = C_i \oplus C_i \oplus (\ker d_i \otimes \kmw_n)$. Computing  image of $\partial_{i+1}$ yields~$\Img \partial_i = C_i \oplus C_i \oplus (\Img d_{i+1}\otimes \kmw_n)$. The homology of the complex is $ \mathbf{H}_i(C)\otimes \kmw_n$ for $i\geq 1$. 
		\end{proof}
		A similar proof strategy yields the following result.
		\begin{lemma}
			\label{lem:homologySusp} 
            Let $C$ be a pointed chain complex of strictly $\AAA$-invariant sheaves, then there is an isomorphism
            \begin{equation*}
                \mathbf{H}_i(\Sigma C) \cong \begin{cases}
                \Z  & i=0,\\
                    \widetilde{\mathbf{H}}_{i-1}(C)  & i > 0.
                \end{cases}
            \end{equation*}
		\end{lemma}
		
  Applying Lemmas \ref{lem:homologyGMsmash} and \ref{lem:homologySusp} to Corollary \ref{cor:homology} yields the following theorem. 
        \begin{theorem}
        \label{thm:cellularHomology}
			Let $K$ be a simplicial complex. Then $\mathbf{H}_0^{\mathrm{cell}}(\zka)=\Z$ and for $i>0$ 
            \begin{equation*}
                \mathbf{H}_i^{\mathrm{cell}}(\zka)\cong\bigoplus_{I \not \in K} \widetilde{\mathbf{H}}_{i-1}(|K_I|)\otimes \kmw_{|I|}.
            \end{equation*}
		\end{theorem}
  
\begin{example}
    Let $K = \partial\Delta^{m-1}$, then  
    \begin{equation*}
         \mathbf{H}_i^{\mathrm{cell}}(\zka)\cong \begin{cases}
             \Z & i=0, \\
             \kmw_{m} & i= m-1, \\
             0 & i\neq 0,m-1.
         \end{cases}
    \end{equation*}
\end{example}
\begin{example}
\label{ex:RP2TriangHomology}
    Let $K$ be the following triangulation of $\R \PP^2$. Vertices with the same labels are identified, and all triangles are filled in.
    
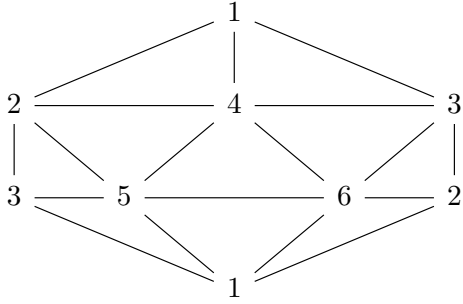
\begin{equation*}
\begin{tikzcd}
          &            & 1 \arrow[rrd, no head] \arrow[d, no head] &                                          &                                          \\
2 \arrow[rr, no head] \arrow[rru, no head] \arrow[d, no head] \arrow[rd, no head] &                                                               & 4 \arrow[rd, no head] \arrow[rr, no head] &                                          & 3 \arrow[d, no head] \arrow[ld, no head] \\
3 \arrow[r, no head] \arrow[rrd, no head]                                         & 5 \arrow[rr, no head] \arrow[ru, no head] \arrow[rd, no head] &                                           & 6 \arrow[r, no head] \arrow[ld, no head] & 2 \arrow[lld, no head]                   \\
     &    & 1    &     &  
\end{tikzcd}
    \end{equation*}
    By Theorem \ref{thm:cellularHomology}, we get the following decomposition of the cellular $\AAA$-homology of~$\zka$.
    \begin{equation*}
        \mathbf{H}_i^{\mathrm{cell}}(\zka) = \begin{cases}
            \Z & i = 0, \\ 0 & i =1, \\ (\kmw_{3})^{\oplus 10} \oplus (\kmw_{4})^{\oplus 15} \oplus (\kmw_{5})^{\oplus 6} \oplus  (\Z_2 \otimes \kmw_{6}) & i=2, \\
            0 & i \geq 3.
        \end{cases}
    \end{equation*}
    This is an example of a space with integral torsion in its cellular $\AAA$-homology.
\end{example}

		\begin{remark}
		    Theorem \ref{thm:cellularHomology} could be computed using Corollary \ref{cor:stableSplittingContractible} as follows. By Corollary \ref{cor:stableSplittingContractible} there is an equivalence
       \begin{equation*}
          \Sigma (C_*^{\mathrm{cell}}(\AAA),C_*^{\mathrm{cell}}(\Gm))^K \simeq  \bigvee_{I\not \in K} \Sigma^2 \B{K_I}\wedge C_*^{\mathrm{cell}}(\Gm) ^{\wedge |I|}.
      \end{equation*}Since Lemma \ref{lem:homologySusp} tells us how the homology changes after suspending and Lemma \ref{lem:homologyGMsmash} in the case of $n=1$ tells us what smashing with $C_*^{\mathrm{cell}}(\Gm)$ does to the homology of a chain complex. Computing the homology and accounting for the extra suspension recovers Theorem \ref{thm:cellularHomology}.
		\end{remark}
		
		\subsection{Cohomology}
        Motivic cohomology is a bigraded cohomology theory for $\calH(k)$. For a motivic space~$X$ and a commutative ring $\mathbf{k}$, write $H^{i,j}_{\mathrm{Mot}}(X;\mathbf{k})$ for the motivic cohomology of $X$ with coefficients in $\mathbf{k}$. We denote the cohomology of the point by $A=H^{*,*}_{\mathrm{Mot}}(\Spec(k);\mathbf{k})$, thus for $p,q \in \Z$ we have $A^{p,q}=H^{p,q}_{\mathrm{Mot}}(\Spec(k);\mathbf{k})$. For a bigraded $A$-module $M$, and two integers $i,j \in \Z$ we define $M[i,j]$ to be $M$, but with the grading shifted to by $(i,j)$, that is $M[i,j]^{p,q}=M^{p-i,q-j}$. With this, the cohomology of motivic spheres can be described as an $A$-module
        \begin{equation*}
            H^{*,*}_{\mathrm{Mot}}(S^{p,q};\mathbf{k})\cong A \oplus A[p,q].
        \end{equation*}
        We may use the stable splitting of $\zka$ to describe the cohomology groups of $\zka$ for certain simplicial complexes $K$. We denote the reduced motivic cohomology of a motivic space $X$ by $\widetilde{H}^{i,j}_{\mathrm{Mot}}(X;\mathbf{k})$.
 
        \begin{theorem}
            \label{thm:zkacohomology}
            Let $K$ be a simplicial complex. The reduced motivic cohomology groups of $\zka$ are given by the isomorphism
            \begin{equation*}
                \widetilde{H}^{p,q}_{\mathrm{Mot}}(\zka;\mathbf{k}) \cong \bigoplus_{I\not\in K, |I|=j} \widetilde{H}^{p-j-1,q-j}_{\mathrm{Mot}}(\B{K_I};\mathbf{k}).
            \end{equation*}
        \end{theorem}
        \begin{proof}
            For a motivic space $X$, there is the relation $\widetilde{H}^{p,q}_{\mathrm{Mot}}(X ;\mathbf{k})=\widetilde{H}^{p+i,q+j}_{\mathrm{Mot}}( X\wedge S^{i,j};\mathbf{k})$. 
            Combining this with the stable splitting from Theorem \ref{thm:stableMMACsplit} yields the desired result.
        \end{proof}
         In the case where $\Sigma |K_I|$ (as a topological space) is a wedge of spheres, we can express the cohomology of $\zka$ just in terms of shifted copies of $A$.
         \begin{proposition}
         \label{prop:cohomologymodule-decomp}
             Let $K$ be a simplicial complex on the vertex set $[m]$ such that $\Sigma |K_I|$ is a wedge of spheres for all $I \not \in K$. Then there is an isomorphism of $A$-modules
             \begin{equation*}
                \widetilde{H}^{*,*}_{\mathrm{Mot}}(\zka;\mathbf{k}) \cong \bigoplus_{\substack{I \not \in K, |I|=j, \\ 0\leq i \leq m-2}} A[i+j+1,j]^{\oplus \mathrm{rank } \widetilde{H}^{i-j-1}{(K_I;\Z)}}.
             \end{equation*}
             
        \end{proposition}
        \begin{proof}
            Fix $I\in K$. If $\Sigma\B{K_I}$ splits into a wedge of spheres, then its homotopy type is solely determined by the rank of its singular homology groups. This allows us to express the motivic cohomology of $\B{K_I}$ in terms of $A$.
            \begin{equation*}
                \widetilde{H}^{*,*}_{\mathrm{Mot}}(\B{K_I} ;\mathbf{k}) = \bigoplus_{i = 0}^{m-2} A[i,0]^{\oplus \mathrm{rank } \widetilde{H}^{i}{(K_I;\Z)}}
            \end{equation*}
            By summing over each $I \not \in K$ and shifting according to suspensions by $\Gm$, we recover the result. The reason the sum is not infinite is because $K$ is a simplicial complex is finite dimensional.        
        \end{proof}
        \begin{remark}
            Simplicial complexes such as flag complexes and triangulations of spheres satisfy the conditions of Proposition \ref{prop:cohomologymodule-decomp}.
        \end{remark}
		Classically, figuring out the ring structure of the cohomology of the moment-angle complex can be done with the Eilenberg--Moore spectral sequence. Unfortunately, we run into some problems when using the same approach motivically as there is no suitable Eilenberg--Moore spectral sequence available for us to use. 

        \begin{remark}
            There is a version of the Eilenberg-Moore spectral sequence due to Krishna \cite{Krishna}, but it only computes cohomology groups. In particular, Krishna provides for each integer $j$, a spectral sequences converging to $H^{*,j}_{\mathrm{Mot}}$.
        \end{remark}

        When the base field $k=\C$ we are able to say some things due to complex realization. Work by Levine \cite{LevineComparison} shows that complex realization is a symmetric monoidal functor from the stable motivic homotopy category to the stable homotopy category (of topological spaces). It follows  that complex realization induces for any commutative ring $\mathbf{k}$ a $\mathbf{k}$-algebra homomorphism
        \begin{equation*}
           \phi \colon \bigoplus_{i,j} H^{i,j}_{\mathrm{Mot}}(X;\mathbf{k}) \to \bigoplus_i H^{i}(X(\C);\mathbf{k}).
        \end{equation*}
        Since $\phi$ is a $\mathbf{k}$-algebra morphism lets us pull back cup products from~$H^{i}(X(\C);\mathbf{k})$ to~$H^{i,j}_{\mathrm{Mot}}(X;\mathbf{k})$. That is, if there exists $\alpha, \beta \in H^{*,*}_{\mathrm{Mot}}(\zka;\mathbf{k})$ such that~${\phi(\alpha)\smile \phi(\beta) \neq 0}$, then there exists $\gamma \in H^{*,*}_{\mathrm{Mot}}(\zka;\mathbf{k})$ such that~${\alpha \smile \beta = \gamma}$.
		
		\subsection{Betti numbers}
		A much coarser invariant than (co)homology are Betti numbers. By Theorem \ref{thm:TorIso} the cohomology of the moment-angle complex $\zk = (D^2,S^1)^K$ is isomorphic to a bigraded Tor-algebra   \begin{equation*}
            H^{2j-i}(\zk) \cong  \text{Tor}^{i,j}_{\Z[v_1,\ldots,v_n]}(\Z,\Z[K]).
        \end{equation*}
        The \emph{bigraded Betti numbers $b^{i,j}$} of $\zk$ are defined as follows
        \begin{equation*}
            b^{i,j}(\zk) := \text{rank } \text{Tor}^{i,j}_{\Z[v_1,\ldots,v_n]}(\Z,\Z[K]).
        \end{equation*}
		When $K$ is a simplicial complex where $\Sigma \B{K_I}$ is a wedge of spheres for all $I$, the description of the cohomology of $\zka$ from Proposition \ref{prop:cohomologymodule-decomp} allows us to define Betti numbers. \begin{definition}
            
		    The $(i,j)$th $\AAA$-Betti number of $\zka$ is defined as follows
        \begin{equation*}
            b_{\AAA}^{i,j}(\zka) := \begin{cases}
            \sum_{I \not \in K, |I|=j} \text{rank } \widetilde{H}^{i-j-1}{(K_I;\Z)} & (i,j) \neq (0,0), \\
                1 & (i,j)=(0,0).\end{cases}
        \end{equation*}
        The following example highlights the choice of definition. 
        \begin{example}
            Let $K =\partial \Delta ^n$, then $\zka \simeq S^{2n,n-1}$ and 
            \begin{equation*}
                b_{\AAA}^{i,j}(\zka) = \begin{cases}
                    1 & (i,j) = (0,0) \text { or } (i,j)=(2n,n-1), \\
                    0 & \text{otherwise}.
                \end{cases}
            \end{equation*}
        \end{example}
		\end{definition} 
		We now have two different bigraded Betti numbers related to (motivic) moment-angle complexes. The next step is to compare the two notions. 
		\begin{theorem}
  Let $K$ be a simplicial complex. Then there is an equality
			\begin{equation*}
				b_{\AAA}^{i,j}(\zka) = b^{-j,2i}(\zk).
			\end{equation*}
		\end{theorem}
		\begin{proof}
			In the classical case, we have the following well known isomorphism of groups due to Hochster \cite[Theorem 5.1]{Hochster},
			\begin{equation*}
				\text{Tor}^{-j,2i}_{\Z[v_1,\ldots,v_n]}(\Z,\Z[K])\cong \bigoplus_{I \subset [m] : |I|=i}\widetilde{H}^{i-j-1}(K_I).
			\end{equation*}
            Thus the bigraded Betti numbers of $\zk$ can be expressed as
            \begin{equation*}
                b^{-j,2i}(\zk) = \sum_{I \subset [m] : |I|=i} \text{rank } \widetilde{H}^{i-j-1}(K_I;\Z) =  b_{\AAA}^{i,j}(\zka).\qedhere
            \end{equation*}
		\end{proof}
		\subsection{Euler characteristics}
        For any symmetric monoidal category $\calC$, there is a notion of categorical Euler characteristic of a dualizable object \cite{DoldPuppe}. Let $1_\calC$ be the unit of $\calC$. If $X$ is a dualizable object in $\calC$, then there exists a dual object $X^\vee$, an evaluation map $\epsilon\colon X \otimes X^\vee \to 1_\calC$, and a coevaluation map $\eta \colon 1_\calC \to  X \otimes X^\vee$.
        The categorical Euler characteristic of a dualizable object $X$ is the composition
        \begin{equation*}
            1_\calC \xrightarrow{\eta} X \otimes X^\vee \xrightarrow{\mathrm{id}_X \otimes \mathrm{id}_{X^\vee}} X \otimes X^\vee \xrightarrow{\epsilon} 1_\calC.
        \end{equation*}
        We denote the categorical Euler characteristic by is written as $\chi_\calC(X)$. We see that the categorical Euler characteristic takes values in $\mathrm{End}(1_\calC)$ endomorphism of the unit object of $\calC$.
        
        We denote the stable motivic homotopy category over a field $k$ by $\SH(k)$. In the case of the stable motivic homotopy category $\mathrm{End}(1_{\mathrm{SH}(k)})=\GW(k)$ \cite[Theorem 6.4.1]{MorelIntroduction}, where $\GW(k)$ denotes the Grothendieck--Witt ring of quadratic forms over $k$. The elements of $\GW(k)$ are formal differences of $k$-valued, non-degenerate, quadratic forms on finite dimensional $k$-vector spaces. For a unit $u \in k^\times$, we let $\langle u  \rangle \in \GW(k)$ denote the rank one quadratic form $x \mapsto ux^2$. In addition, $\GW(k)$ is generated by the rank one forms as a group. For units $u,v\in k^\times$, we have $\langle u \rangle \cdot \langle v \rangle = \langle uv \rangle$. For any $u \in k^{\times}$, there is an equivalence $\langle u^2 \rangle = \langle 1 \rangle$. Thus $\GW(\C) \cong \Z$ and $\GW(\R) \cong \Z \times \Z$. Examples of computations of categorical Euler characteristics in the stable homotopy category can be found in \cite{LevineComparison, LevineRaksit, HoyoisTraceFormula}.

        In motivic homotopy theory, there is an $\AAA$-Euler characteristic closely related to~$\chi_{\SH(k)}$. We say that a motivic space $X \in \Hk$ is \emph{dualizable} if the $\PP^1$-suspension spectrum~${\infsusp X_+}$ is dualizable in $\SH(k)$. Let $X$ be a dualizable object in $\Hk$, then the $\A^1$-Euler characteristic $\chi_{\A^1}(X)$  is defined as follows \begin{equation*}
			\chi_{\A^1}(X) := \chi_{\SH(k)}(\infsusp X_+).
		\end{equation*}
        For a finite  $CW$-space $X$, let $\chi(X)$ denote its Euler characteristic. Real and complex realization allows us to relate the $\AAA$-Euler characteristic to its classical counterpart. When~$k$ is a subfield of $\R$, we have the following relation between the $\AAA$-Euler characteristic and the Euler characteristic of the real and complex points of $X$ \cite[Remark~1.3] {Levine2020}. 
        \begin{equation*}
            \chi (X(\C)) = \text{rank } \chi_{\AAA}(X) \quad \text{and} \quad \chi (X(\R)) = \text{sign } \chi_{\AAA}(X).
        \end{equation*}
        If $k$ is just a subfield of $\C$, we only have the first relation for the complex points. 
        \begin{example}
            Recall that $\G_m(\C) \simeq S^1$ and $\G_m(\R) \simeq S^0$. Let $k$ be a field of characteristic different from $2$. Then~${\chi_{\AAA}(\Gm)= \langle 1 \rangle - \langle -1 \rangle}$. When $k=\C$, all units are squares. Thus $\langle -1 \rangle = \langle 1 \rangle$, and $\chi_{\AAA}(\Gm) = 0 = \chi(S^1)$. When $k=\R$, the signature of~$\langle 1 \rangle - \langle -1 \rangle$ is $2$ which coincides with $\chi(S^0)$.

        \end{example}
         The Euler characteristic of a moment-angle complex is not an interesting invariant, because by Example \ref{ex:EulerCharEx} $\chi(\zk)=0$ for any $K$. However, the real moment-angle complex can have nonzero Euler characteristic. The realization result above thus suggests that the motivic moment-angle complex does not have trivial Euler characteristic when~$k$ is a subfield of the real numbers. 

        The first thing we need is to show that $\zka$ admits an $\AAA$-Euler characteristic. 
        \begin{lemma}
            Let $K$ be a simplicial complex. Then $\zka$ is dualizable in $\calH(k)$.
        \end{lemma}
        \begin{proof}
            Using Theorem \ref{thm:stableMMACsplit}, the  $\PP^1$-suspension spectrum of $\zka$ can be written as 
            \begin{equation*}
                \infsusp {\zka}_+ = \infsusp(S^{0,0})\vee \Sigma^{-1,0} \bigvee_{I \not \in K} \Sigma^{|I|+2,|I|}\infsusp\B{K_I}.
            \end{equation*}
            Let $\SH$ be the stable homotopy category of topological spaces. There is a symmetric monoidal functor $\SH \to \SH(k)$, which in particular preserves dualizable objects. Since the spaces $\infsusp \B{K_I}$ are images of finite $CW$-spectra under this map, they are also dualizable. Since spheres and wedges of dualizable spaces are dualizable, $\zka$ is dualizable. 
        \end{proof}
        Earlier, we saw that Davis had computed the Euler characteristic of polyhedral products (Theorem \ref{thm:DavisEuler}). We can recover the result for motivic moment-angle complexes. 
        \begin{theorem} 
        \label{thm:DavisA1EulerChar}
			The $\A^1$-Euler characteristic of the motivic moment-angle complex is 
			\begin{equation*}
				\chi_{\A^1}(\zka)=\sum_{\sigma \in K}\langle -1 \rangle^{|\sigma|}(\langle 1 \rangle-\langle -1 \rangle)^{m-|\sigma|}=\sum_{\sigma \in K} (-1)^{|\sigma|}2^{m-|\sigma|-1}(\langle 1 \rangle-\langle -1 \rangle).
			\end{equation*}
		\end{theorem}
		\begin{proof}
		 By using \cite[Lemma 1.4(3)]{Levine2020}, we can express the $\AAA$-Euler characteristic of a smooth scheme in terms of an open subscheme and its complement. By applying this to the cellular structure of $\zka$ from Proposition \ref{prop:cellularStructure}, we recover the theorem. 
		\end{proof}
  We will also give a computation of $\chi_{\AAA}(\zka)$ that uses the stable splitting, and hence describes the $\AAA$-Euler characteristic in terms of the Euler characteristic of full subcomplexes of $K$. For a simplicial complex $K$, we write $\chi(K)$ for it classical Euler characteristic. 
		\begin{theorem}
			Let $K$ be a simplicial complex. \label{thm:A1EulerChar}The $\A^1$-Euler characteristic of the motivic moment-angle complex is 
			\begin{equation*}
				\chi_{\A^1}(\zka) =  \langle 1 \rangle - \sum_{I \not \in K} (-1)^{|I|}(\chi(K_I)-1)\cdot \langle -1 \rangle^{|I|}.
			\end{equation*}
		\end{theorem}
  We postpone the proof of the theorem until the end of the section.
  \begin{example}
      Let $K$ be a square, as in Example \ref{ex:squareCplx}. Then $K$ has four vertices and four edges. Theorem \ref{thm:DavisA1EulerChar}, we compute
      \begin{equation*}
          \chi_{\A^1}(\zka) = (2^3-4\cdot 2^2+4\cdot 2)\cdot(\langle 1 \rangle-\langle -1 \rangle) = 0.
      \end{equation*}
      Using Theorem \ref{thm:A1EulerChar}, we have three full subcomplexes of $K$ corresponding to the cases where $I =\{1,2\}, \{3,4\}$ or $\{1,2,3,4\}$. 
      \begin{equation*}
           \chi_{\A^1}(\zka) = \langle 1 \rangle - 2(\chi(S^0)-1)\cdot\langle 1 \rangle -(\chi(S^1)-1) \cdot\langle 1 \rangle = (1-2+1)\langle 1 \rangle = 0.
      \end{equation*}
  \end{example}
		 The $\A^1$-Euler characteristic also exhibits some nice properties like the classical Euler characteristic.
		\begin{lemma}
			\label{lem:eulCharWedge}
			$\chi_{\A^1}(X\vee Y) = \chi_{\A^1}(X) + \chi_{\A^1}(Y) - \langle 1 \rangle$.
		\end{lemma}
		\begin{proof}
			The wedge $X\vee Y$ may be written as the homotopy pushout 
			\begin{equation*}
                \begin{tikzcd}
        * \arrow[dr, phantom, "\scalebox{1}{$\ulcorner$}" , very near end, color=black]\ar[r]\ar[d] & X \ar[d] \\
        Y \ar[r] & X\vee Y  .
    \end{tikzcd}
            \end{equation*}
			In \cite{May2001} May proved that the following holds for the categorical Euler characteristic.
			\begin{equation*}
				\chi_{\SH(k)}(\infsusp X\vee Y) = \chi_{\SH(k)}(\infsusp X) + \chi_{\SH(k)}(\infsusp Y) - \chi_{\SH(k)}(\infsusp \Spec(k)) 
			\end{equation*}
			The diagram above is still a homotopy pushout after adding a disjoint basepoint, thus we get 
			
			\begin{equation*}
				\chi_{\A^1}(X\vee Y) = \chi_{\A^1}(X) + \chi_{\A^1}(Y) - \langle 1 \rangle. \qedhere
			\end{equation*}
		\end{proof}
		
		In classical topology we have the following relation between the Euler characteristic of a space $X$ and its suspension $\Sigma X$
		\begin{equation*}
			\chi (\Sigma X) = 2 - \chi(X).
		\end{equation*}
		Similar to the classical relation, we have the following relation of $\A^1$-Euler characteristics. 
		
		\begin{lemma}
			\label{lem:eulCharSmash}
			$\chi_{\A^1}(X \wedge S^{p,q}) = \langle 1 \rangle + (-1)^p\langle -1 \rangle^q(\chi_{\A^1}(X)-\langle 1 \rangle).$
		\end{lemma}
		\begin{proof}
  We have
			\begin{equation*}
				\chi_{\A^1}(X \wedge S^{p,q}) = \chi_{\SH(k)}(\infsusp (X \wedge S^{p,q})) + \langle 1 \rangle.
			\end{equation*}
			For the categorical Euler characteristic, May \cite{May2001} proved that \begin{equation*}
				\chi_{\SH(k)}(\infsusp X \wedge S^{p,q}) = \chi_{\SH(k)}(\infsusp X)\cdot\chi_{\SH(k)}(\infsusp S^{p,q}).
			\end{equation*}
			We may rewrite $\chi_{\SH(k)}(\infsusp X) = \chi_{\A^1}(X)-\langle 1\rangle $ and use the $\chi_{\SH(k)}(\infsusp S^{p,q}) = (-1)^p\langle -1 \rangle^q$ by \cite[Lemma 1.2]{Levine2020} to get the claimed result. 
		\end{proof}
		We can now prove Theorem \ref{thm:A1EulerChar}.
		\begin{proof}[Proof of Theorem \ref{thm:A1EulerChar}]
			By Lemma \ref{lem:eulCharSmash}, the $\A^1$-Euler characteristic of $\zka$ can be expressed in the following way\begin{equation*}
				\chi_{\A^1}(\zka) = 2\langle 1 \rangle - \chi_{\AAA}(\Sigma \zka).
			\end{equation*}
			Since $\Sigma\zka$ splits into a wedge sum by Theorem \ref{thm:stableMMACsplit}, we get 
			\begin{align*}
				\chi_{\A^1}(\zka) &= 2\langle 1 \rangle - \chi_{\A^1}(\bigvee_{I \not \in K} \Sigma^2 \B{K_I}\wedge \Gm^{\wedge |I|}) \\
				&= 2\langle 1 \rangle - \chi_{\A^1}(\bigvee_{I \not \in K} \B{K_I}\wedge S^{|I|+2,|I|}).
			\end{align*}
			We continue by applying Lemma \ref{lem:eulCharWedge} to the wedge sum resulting in
			\begin{equation*}
				\chi_{\A^1}(\zka) = \langle 1 \rangle - \sum_{I \not \in K}\left(\chi_{\A^1}(\B{K_I}\wedge S^{|I|+2,|I|}) - \langle 1 \rangle \right).
			\end{equation*}
			We then apply Lemma \ref{lem:eulCharSmash} to get
			\begin{equation*}
				\chi_{\A^1}(\zka) = \langle 1 \rangle - \sum_{I \not \in K} (-1)^{|I|}\langle -1 \rangle^{|I|}(\chi_{\A^1}(\B{K_I})-\langle 1 \rangle).
			\end{equation*}
			For any simplicial complex $K$, we have $\chi_{\A^1}(\B{K}) = \chi(K)\cdot \langle 1 \rangle$. This allows us to rewrite the result above to
			\begin{equation*}
				\chi_{\A^1}(\zka) = \langle 1 \rangle - \sum_{I \not \in K} (-1)^{|I|}(\chi(K_I)-1)\cdot \langle -1 \rangle^{|I|}. \qedhere
			\end{equation*}
			
		\end{proof}

\addcontentsline{toc}{chapter}{Bibliography}

\bibliographystyle{abbrv}

\bibliography{mylib}

\cleardoublepage

	\end{document}